\documentclass[12pt]{elsarticle}

\usepackage{mystyleElsArticle}
\usepackage{graphicx}
\usepackage{diagbox}
\usepackage{mathtools}
\usepackage{color}
\usepackage{tikz}
\usetikzlibrary{matrix}
\makeatletter
\def\ps@pprintTitle{%
  \let\@oddhead\@empty
  \let\@evenhead\@empty
  \let\@oddfoot\@empty
  \let\@evenfoot\@oddfoot
}
\makeatother
\usepackage[margin=1.0in]{geometry}    

\begin{document}

\begin{frontmatter}
  \title{
  Computing effective diffusivity of chaotic and stochastic flows using structure preserving schemes}


   \author[hku]{Zhongjian Wang}
   \ead{ariswang@connect.hku.hk}
   \author[uci]{Jack Xin}
   \ead{jxin@math.uci.edu}
   \author[hku]{Zhiwen Zhang\corref{cor1}}
   \ead{zhangzw@hku.hk}

   \address[hku]{Department of Mathematics, The University of Hong Kong, Pokfulam Road, Hong Kong SAR.}
   \address[uci]{Department of Mathematics, University of California at Irvine, Irvine, CA 92697, USA.}

   \cortext[cor1]{Corresponding author}

\begin{abstract}
In this paper we study the problem of computing the effective diffusivity for a particle moving in chaotic and stochastic flows.
In addition we numerically investigate the residual diffusion phenomenon in chaotic advection. The residual diffusion refers to the non-zero effective (homogenized) diffusion in the limit of zero molecular diffusion as a result of chaotic mixing of the streamlines. In this limit traditional numerical methods typically fail since the solutions of the advection-diffusion equation develop sharp gradients. Instead of solving the Fokker-Planck equation in the Eulerian formulation, we compute the motion of particles in the Lagrangian formulation, which is modelled by stochastic differential equations (SDEs).  We propose a new numerical integrator based on a stochastic splitting method to solve the corresponding SDEs in which the deterministic subproblem is symplectic preserving while the random subproblem can be viewed as a perturbation. We provide rigorous error analysis for the new numerical integrator using the backward error analysis technique and show that our method outperforms
standard Euler-based integrators.  Numerical results are presented to demonstrate the accuracy and efficiency of the proposed method for several typical chaotic and stochastic flow problems of physical interests. \\
\textit{\textbf{AMS subject classification:}} 76R99, 35B27, 65T40, 65M70
\end{abstract}

\begin{keyword}
Advection-diffusion; chaotic flows; stochastic flows; effective diffusivity;  structure preserving schemes;
stochastic Hamiltonian systems; backward error analysis.
\end{keyword}
\end{frontmatter}


\section{Introduction} \label{sec:introduction}
\noindent
Diffusion enhancement in fluid advection is a fundamental problem to characterize and quantify the large-scale effective diffusion
in fluid flows containing complex and turbulent streamlines, which is of great theoretical and practical importance, see  \cite{Fannjiang:94,Fannjiang:97,Biferale:95,Majda:99,
PavliotisStuart:03,PavliotisStuart:05,PavliotisStuart:07,StuartZygalakis:09,JackXin:11,JackXin:15} and references therein. Its applications can be found in many physical and engineering sciences, including atmosphere/ocean science, chemical engineering, and combustion. In this paper, we shall study a passive tracer model, which describes particle
motion with zero inertia:
\begin{align}
\dot{X}(t) = v(X,t) + \sigma \xi(t),  \quad  X\in R^{d},   \label{passivemodel}
\end{align}
where $X$ is the particle position, $\sigma\geq0$ is the molecular diffusion coefficient, and $\xi(t)\in R^d$ is a white noise or colored noise. The velocity $v(x,t)$ satisfies either the Euler or the Navier-Stokes equation. We point out that in practice, $v(x,t)$ can be modeled by a random field which mimics energy spectra of the velocity fields. We set $v(x,t) = \nabla^{\bot} \phi(x,t)$ and
the streamline function $\phi$ satisfies $\phi_{t} = A\phi + \sqrt{Q}\zeta(x,t)$, which is a random field generated by appropriately choosing operators $A$ and $Q$ and $\zeta(x,t)$ is a space-time white noise independent of $\xi(t)$. This will be investigated in our subsequent paper.

For spatial-temporal periodic velocity fields and random velocity fields with short-range correlations,
the homogenization theory \cite{Bensoussan:78,Garnier:97,Oleinik:94,Stuart:08} says that the long-time large-scale behavior of the particles is governed by a Brownian motion. More precisely, let $D^{E}\in R^{d\times d}$ denote the effective diffusivity matrix and
$X^{\epsilon}(t)\equiv\epsilon X(t/\epsilon^2)$. Then, $X^{\epsilon}(t)$ converges in distribution to a Brownian motion $W(t)$ with covariance matrix $D^{E}$, i.e., $X^{\epsilon}(t) \xrightarrow{\text{d}}\sqrt{2D^E}W(t)$. The $D^E$ can be expressed in terms of particle ensemble average (Lagrangian framework) or cell problems (Eulerian framework). The dependence of $D^E$ on the velocity field of the problem is highly nontrivial. For time-independent Taylor-Green velocity field, the authors of \cite{StuartZygalakis:09} proposed a stochastic splitting method and calculated the effective diffusivity in the limit of vanishing molecular diffusion. For random velocity fields with long-range correlations, various forms of anomalous diffusion, such as super-diffusion and sub-diffusion, can be obtained for some exactly solvable models (see \cite{Majda:99} for a review). However, long-time large-scale behavior is in general difficult analytically.

This motivates us to study numerically the dependence of $D^E$ on complicated incompressible, time-dependent velocity fields in this paper. We are also interested in investigating the existence of residual diffusivity for the passive tracer model Eq.\eqref{passivemodel} for several different velocity fields. The residual diffusivity refers to the non-zero effective diffusivity in the limit of zero molecular diffusion as a result of a fully chaotic mixing of the streamlines. It is expected that the corresponding long-time large-scale behavior will follow a different law and sensitively depend on the velocity fields. In \cite{JackXinLyu:2017}, the authors solved computed the cell problem of the advection-diffusion type and observed the residual diffusion phenomenon. This approach allows adaptive basis learning for parameterized flows. However, the solutions of the advection-diffusion equation develop sharp gradients as molecular diffusion approaches zero and demand a large amount of computational costs in standard Fourier basis. To overcome this difficulty, we shall adopt the Lagrangian framework and compute
an ensemble of particles governed by Eq.\eqref{passivemodel} directly.

In this paper, we shall compute the effective diffusivity of stochastic flows using structure preserving schemes and investigate the existence of residual diffusivity for several prototype velocity fields. First, we propose a new numerical integrator based on a stochastic splitting method to solve the corresponding SDEs in which the deterministic subproblem is symplectic preserving while the random subproblem can be viewed as a perturbation.
Then using the backward error analysis (BEA)\cite{Reich:99}, we prove that our numerical integrator preserve the invariant measure on torus space (the original space moduled by its space-time period), while the standard Euler-based integrator does not have this property. Thus, our method is capable of computing long-time behaviors of the passive tracer model.

The rest of the paper is organized as follows. In Section 2, we give a brief introduction of the background of the passive tracer model
and derivation of the effective diffusivity tensor using multiscale technique. In section 3, we propose our new method for computing the
passive tracer model. Error estimate of the proposed method will be discussed in Section 4. We use the BEA technique and find that for a class of separable Hamiltonian our method preserves the structure and achieves an asymptotically convergence to effective diffusivity. Issues regarding the practical implementation of our method will also be discussed. In Section 5, we present numerical results to demonstrate the accuracy and efficiency of our method. We also investigate the existence of residual diffusivity for time periodic and stochastic velocity fields. Concluding remarks are made in Section 6.

\section{Effective diffusivity and multiscale technique} \label{sec:EffectiveDiffusivity}
\noindent
We first give a brief introduction of the effective diffusivity for stochastic flows. The motion of a particle
in a velocity field can be described by the following stochastic differential equation,
\begin{align}
\dot{X}(t) = v(X,t) + \sigma \xi(t),  \quad  X\in R^{d},   \label{eqn:generalSDE}
\end{align}
where $\sigma$ is the molecular diffusion, $X$ is the position of the particle, $v(X,t)$ is the Eulerian
velocity field at position $X$ and time $t$, $\xi(t)$ is a Gaussian white noise with zero mean and correlation function
$<\xi_{i}(t)\xi_{j}(t^{\prime})>=\delta_{ij}\delta(t-t^{\prime})$. Here $\langle\cdot\rangle$ denotes ensemble average over all randomness.

Given any initial density $u_{0}(x)$, particle $X(t)$ of Eq.\eqref{eqn:generalSDE} has a density $u(x,t)$ which is given by the Fokker-Planck equation,
\begin{align}
u_t + \nabla\cdot (vu )  = D_0 \Delta u, \quad u(x,0)= u_{0}(x), \quad  x\in R^{d},  \label{FokkerPlanck-eq}
\end{align}
where $D_0=\sigma^2/2$ is the diffusion coefficient. When $v(x,t)$ is incompressible (i.e. $\nabla_x v(x,t)=0\quad \forall t$), deterministic and space-time \emph{periodic} in $O(1)$ scale (for convenience we assume the period of $v$ in space  is $1$ and in time is $T_{per}$),
the formula for the effective diffusivity tensor is \cite{Bensoussan:78,Biferale:95}
\begin{align}
D_{ij}^{E} =D_0\big(\delta_{ij}+\langle \nabla w_i \cdot \nabla w_j \rangle_{p}\big),
\label{Def_EffectiveDiffusivity_Euler}
\end{align}
where $w(x,t)\in \mathbb{T}^{d}\times[0,T_{per}]$ is the periodic solution of the cell problem
\begin{align}
w_t + v \cdot \nabla w - D_0\triangle w = -v, \label{CellProblem_EffectiveDiffusivity}
\end{align}
and $\langle \cdot \rangle_{p} $ denotes space-time average over periods.
As $v$ is incompressible, solution $w(x,t)$ of the cell problem Eq.\eqref{CellProblem_EffectiveDiffusivity} is unique up to an additive constant by the Fredholm
alternative. The correction to $D_0$ is positive definite in Eq.\eqref{Def_EffectiveDiffusivity_Euler}. In practice,
the cell problem Eq.\eqref{CellProblem_EffectiveDiffusivity} can be solved using numerical methods,
such as spectral methods. In \cite{JackXinLyu:2017}, a small set of adaptive basis functions were constructed from
fully resolved spectral solutions to reduce the computation cost. However, when $D_0$ becomes extremely small,
the solutions of the advection-diffusion equation Eq.\eqref{CellProblem_EffectiveDiffusivity} develop sharp gradients and demand a large number of Fourier modes to resolve, which makes the Eulerian framework computationally expensive and unstable.

In this paper, we shall investigate the Lagrangian approach to compute the effective diffusivity tensor, which is defined by (equivalent to Eq.\eqref{Def_EffectiveDiffusivity_Euler} via homogenisation theory)
\begin{align}
D_{ij}^{E}=\lim_{t\rightarrow\infty}\frac{\Big\langle\big(x_i(t)-x_i(0))(x_j(t)-x_j(0)\big)\Big\rangle_r}{2t},
\quad 1\leq i,j \leq d,
\label{Def_EffectiveDiffusivity_Lagrangian}
\end{align}
where $X(t)=(x_1(t),...,x_d(t))^{T}$ is the position of a particle tracer at time $t$ and the average $\langle\cdot\rangle_r$ is taken over an ensemble of test particles.
If the above limit exists, that means the transport of the particle is a standard diffusion process, at least on a long-time scale.
This is the typical situation, i.e., the spreading of the particle $\Big\langle\big(x_i(t)-x_i(0))(x_j(t)-x_j(0)\big)\Big\rangle_r$ grows linearly with respect to time $t$, for example when the velocity field is given by the Taylor-Green velocity field \cite{StuartZygalakis:09}. However, there are also cases showing that the spreading of particles does not grow linearly with time but has a power law $t^{\gamma}$, where $\gamma>1$ and $\gamma<1$ correspond to super-diffusive and sub-diffusive behaviors, respectively \cite{Biferale:95,Majda:99}.

The major difficulties in solving Eq.\eqref{eqn:generalSDE} come from two components: (1) the computational time should be long enough to approach the diffusive time scale, and (2) the chaotic and stochastic velocity may increase the dimension of the solution space. To address these issues, we shall develop robust numerical integrators, which are structure-preserving and accurate for long-time integration. In addition, we shall investigate the relationship between several typical  time-dependent velocity fields $v(x,t)$ (including both chaotic and stochastic flows) and the corresponding effective diffusivity in this paper.

\section{New stochastic integrators} \label{sec:NewIntegrators}
\noindent
In this section, we construct the new stochastic integrators for the passive tracer model, which is based on the
operator splitting methods \cite{strang:68,Quispel:02}.  We consider the following \emph{two-dimensional} model problems to illustrate the main idea and emphasize that our method can be used to solve high-dimensional problems without any difficulty,
\begin{align}\label{eqn:HamiltonianSDE}
\begin{cases}
dx_1=v_1(t,x_1,x_2)dt+\sigma_1 dW_1, \quad x_1(0)=x_{10}, \\
dx_2=v_2(t,x_1,x_2)dt+\sigma_2 dW_2, \quad x_2(0)=x_{20}.
\end{cases}
\end{align}
Furthermore, we assume that there exists a Hamiltonian function $H(t,x_1,x_2)$ such that
\begin{align}\label{HamiltonianFunction}
v_1(t,x_1,x_2) = -\frac{\partial H(t,x_1,x_2)}{\partial x_2}, \quad
v_2(t,x_1,x_2) =  \frac{\partial H(t,x_1,x_2)}{\partial x_1}.
\end{align}
In this paper we assume that the Hamiltonian $H(t,x_1,x_2)$ is sufficiently smooth and that first order
derivatives of $v_i(t,x_1,x_2)$, $i=1,2$ are bounded. These conditions are necessary to guarantee the
existence and uniqueness of solutions of Eq.\eqref{eqn:HamiltonianSDE}, see \cite{Oksendal:13}. Moreover, the
boundedness of some higher order derivatives of $v_i(t,x_1,x_2)$ is required when we prove the convergence analysis in Section \ref{sec:ConvergenceAnalysis}.

We first rewrite the particle tracer model Eq.\eqref{eqn:HamiltonianSDE} into an abstract form $\dot{X}=\mathcal{L}X$, where
$X=(x_1,x_2)^{T}$. We then split the operator $\mathcal{L}$ into two operators $\mathcal{L}_{i}$, $i=1,2$, where
\begin{align}
\mathcal{L}_{1}:  & \quad dx_1=v_1(t,x_1,x_2)dt, \quad dx_2=v_2(t,x_1,x_2)dt,  \label{OpSp_deterministic} \\
\mathcal{L}_{2}:  & \quad dx_1=\sigma_1 dW_1,  \quad \quad \quad dx_2=\sigma_2 dW_2, \label{OpSp_stochastic}
\end{align}
corresponding to the deterministic part and the stochastic part, respectively. Finally, we apply composition methods to approximate the integrator $\varphi(\tau)=exp(\tau(\mathcal{L}_{1}+\mathcal{L}_{2}))$ generated from Eq.\eqref{eqn:HamiltonianSDE}. Though the operator splitting methods have been successfully applied to various problems, there is limited work on solving SDEs and SPDEs. We refer to \cite{Milstein:02,Owhadi:10} for recent works on Hamiltonian systems with additive noise.

We approximate the integrator $\varphi(\tau)$ by the Lie-Trotter splitting method and get
\begin{align}
\varphi(\tau)=exp(\tau(\mathcal{L}_{1}+\mathcal{L}_{2}))\approx
exp(\tau\mathcal{L}_{1})exp(\tau\mathcal{L}_{2}). \label{Lie-Trotter}
\end{align}
Now we discuss how to discretize the numerical integrator Eq.\eqref{Lie-Trotter}. From time $t=t_k$ to time
$t=t_{k+1}$, where $t_{k+1}=t_{k}+\tau$, $t_0=0$, assuming the solution $(x_1^{k},x_2^{k})^{T}\equiv(x_1(t_k),x_2(t_k))^{T}$
is given, one can solve the subproblems corresponding to $\mathcal{L}_{1}$ and $\mathcal{L}_{2}$ in a small time step $\tau$ to obtain
$(x_1^{k+1},x_2^{k+1})^{T}$. In our numerical method, we discretize the operator $\mathcal{L}_{1}$ by numerical schemes that preserve symplectic structure and the operator $\mathcal{L}_{2}$ by the Milstein scheme \cite{Oksendal:13}, so we obtain the new stochastic integrators for
Eq.\eqref{eqn:HamiltonianSDE} as follows,
\begin{align}\label{NewStochasticIntegrators_Scheme_L1}
\begin{cases}
x_1^{*} = x_1^{k} + \tau v_1\big(t_k+\beta\tau,\alpha x_1^{*}+(1-\alpha)x_1^{k},(1-\alpha)x_2^{*}+\alpha x_2^{k}\big),  \\
x_2^{*} = x_2^{k} + \tau v_2\big(t_k+\beta\tau,\alpha x_1^{*}+(1-\alpha)x_1^{k},(1-\alpha)x_2^{*}+\alpha x_2^{k}\big),
\end{cases}
\end{align}
where the parameters $\alpha,\beta \in [0,1]$ and
\begin{align}\label{NewStochasticIntegrators_Scheme_L2}
\begin{cases}
x_1^{k+1} = x_1^{*} +\sigma_1 \Delta_{k}W_1(\tau),  \\
x_2^{k+1} = x_2^{*} +\sigma_2 \Delta_{k}W_2(\tau),
\end{cases}
\end{align}
with $\Delta_{k}W_i(\tau)=W_i(t_k+\tau)-W_i(t_k)$, $i=1,2$. In practice, each $\Delta_{k}W_i(\tau)$ is an independent random
variable of the form $\sqrt{\tau}\mathcal{N}(0,1)$.

The symplectic-preserving schemes Eq.\eqref{NewStochasticIntegrators_Scheme_L1} are implicit in general.
Compared with explicit schemes, however, they allow us to choose a relatively large time step to compute. In practice,
we find that few steps of Newton iterations are enough to maintain accurate results. Therefore, the computational cost is
controllable. To design adaptive time-stepping method for Eq.\eqref{eqn:HamiltonianSDE} is an interesting issue, which will be studied in our future work.

In general, the second-order Strang splitting \cite{strang:68} is more frequently adopted in application, for which the
integrator $\varphi(\tau)$ is approximated by
\begin{align}
\varphi(\tau)=exp(\tau(\mathcal{L}_{1}+\mathcal{L}_{2}))\approx
exp(\frac{\tau}{2}\mathcal{L}_{2})exp(\tau\mathcal{L}_{1})exp(\frac{\tau}{2}\mathcal{L}_{2}). \label{StrangSplitting}
\end{align}
In fact, the only difference between the Strang splitting method and the Lie-Trotter splitting method is that the first and last
steps are half of the normal step $\tau$. Thus a more accurate method can be implemented in a very simple way. We skip the details in implementing the Strang splitting scheme here as it is straightforward.

We remark that our new stochastic integrators provide an efficient way to investigate the residual diffusivity.
Because we do not need to solve the advection-diffusion equation Eq.\eqref{CellProblem_EffectiveDiffusivity}, which becomes
extremely challenging when $D_0$ is small. Most importantly, symplectic-preserving schemes provide a robust and accurate
numerical integrator for long-time integrations. We shall theoretically and numerically study its performance over
existing numerical integrators, such as Euler schemes, in the subsequent sections.

\section{Convergence analysis}\label{sec:ConvergenceAnalysis}
\noindent
In this section, we shall provide some convergence results.
We prove that a linear growth of the global error can be obtained if we apply our
numerical methods to solve a Hamiltonian system with a separable Hamiltonian. In addition, we shall estimate the numerical error of our method in computing the effective diffusivity. Our analysis is based on the BEA technique \cite{Reich:99}, which is a powerful tool for the study of the long-time behaviors of numerical integrators.

\subsection{Weak Taylor expansion}
\noindent
In our derivation, we use $(p,q$) to denote the position of the particle interchangeably with $(x_1,x_2)$. Thus, the
Hamiltonian system defined by Eq.\eqref{eqn:HamiltonianSDE} is rewritten as
\begin{align}\label{eqn:simplifiedSDE}
\begin{cases}
dp=-H_qdt+\sigma dW_1,\\
dq=H_pdt+\sigma dW_2,
\end{cases}
\end{align}
where $H\equiv H(t,p,q)$ is the Hamiltonian, $\sigma_1 = \sigma_2 = \sigma$ is a positive constant, and $dW_{i}$, $i=1,2$ are two independent Brownian motion processes.
We assume the Hamiltonian system has a separable form \cite{ErnstLubich:06}
\begin{align}
H(t,p,q)=F(t,p)+G(t,q)  \label{SeparableHamiltonian}
\end{align}
with $g \equiv H_q = g(t,q)$ and $f \equiv H_p = f(t,q)$.
\begin{remark}
The  separable Hamiltonian is quite a natural assumption and has many applications in physical and engineering sciences.
For instance,  $H(p,q)=\frac{1}{2}p^T p + U(q)$, where the first term is the kinetic energy and the second one is the potential energy.
\end{remark}
One natural way to study the expectations of the paths for the SDE given by Eq.\eqref{eqn:simplifiedSDE} is to consider its associated backward Kolmogorov equation \cite{Risken:1989}. Specifically, we associate the SDE with a partial differential
operator $\mathcal{L}_0$, which is called the generator of the SDE, also known as the flow operator.  For the Hamiltonian system Eq.\eqref{eqn:simplifiedSDE}, the corresponding backward Kolmolgorov equation associated is given by
\begin{align}\label{BackwardKolmolgorovEquation}
\begin{cases}
\frac{\partial}{\partial t}\phi=\mathcal{L}_0\phi, \\
\phi(x,0)=\phi_0(x),
\end{cases}
\end{align}
where the operator $\mathcal{L}_0$ is given by
\begin{align}
\mathcal{L}_0=-g\partial_p+f\partial_q+\frac{1}{2}\sigma^2\partial_p^2+\frac{1}{2}\sigma^2\partial_q^2.
\label{HamiltonianFlowOperator}
\end{align}
 The probabilistic interpretation of Eq.\eqref{BackwardKolmolgorovEquation} is that given initial data $\phi_0(x)$, the solution of Eq.\eqref{BackwardKolmolgorovEquation}, $\phi(x,t)$, satisfies $\phi(x,t)=E(\phi_0(X_t)|X_0=x)$, where $X_t=(p(t),q(t))$ is the solution to Eq.\eqref{eqn:simplifiedSDE}. We integrate Eq.\eqref{BackwardKolmolgorovEquation} from $t=0$ to $t=\Delta t$ and obtain
\begin{align}\label{BackwardKolmolgorovEquation_OneStep}
\phi(x,\Delta t)= \phi(x,0)+ \mathcal{L}_0 \int_{0}^{\Delta t}\phi(x,s)ds.
\end{align}
Under certain regularity assumptions on the solution $\phi(x,t)$, we have the Taylor expansion
\begin{align}\label{BackwardKolmolgorovEquation_Taylor}
\phi(x,s)= \phi(x,0)+ s\frac{\partial }{\partial s}\phi(x,0) +  \cdot\cdot\cdot +
\frac{s^N}{N!}\frac{\partial^N }{\partial s^N}\phi(x,0) + R_{N}(x,s),
\end{align}
where $R_{N}(x,s)$ is the remainder term in the Taylor expansion. We substitute the Taylor expansion
Eq.\eqref{BackwardKolmolgorovEquation_Taylor} into Eq.\eqref{BackwardKolmolgorovEquation_OneStep} and get
\begin{align}\label{BackwardKolmolgorovEquation_OneStepWithTaylor}
\phi(x,\Delta t)= \phi(x,0)+  \Delta t \mathcal{L}_0 \phi(x,0) +
\sum_{k=1}^{N}\frac{ \Delta t^{k+1}}{(k+1)!} \mathcal{L}_0 \frac{\partial^k }{\partial s^k}\phi(x,0) +O(\Delta t^{N+2}) .
\end{align}
Recall that $\phi(x,0)=\phi_0(x)$ and $\frac{\partial^k }{\partial s^k}\phi(x,0) = \mathcal{L}_0^{k}\phi_0(x)$, we finally obtain
\begin{align}\label{BackwardKolmolgorovEquation_Fianl}
\phi(x,\Delta t)= \phi_0(x)+  \sum_{k=0}^{N}\frac{ \Delta t^{k+1}}{(k+1)!} \mathcal{L}_0^{k+1} \phi_0(x) +O(\Delta t^{N+2}).
\end{align}
The operator $\mathcal{L}_0^{k+1}$ can be computed systematically. For instance, $\mathcal{L}_0$ has 4 terms, then $\mathcal{L}_0^2$ should have at most $4^2=16$ terms.
In this paper, we find that the first order modified equation has already indicated the advantage of the structure preserving scheme. We shall show this in next subsections.
\begin{remark}
Eq.\eqref{BackwardKolmolgorovEquation_Fianl} provides the general framework for us to analyse the truncation error
by numerical methods. Namely, the numerical flow $\phi^{num}(x,\Delta t)=E[\phi_0(X^{num,k}_{\Delta t})|X_0=x]$ generated by some $k$-th order weak method
should satisfy Eq.\eqref{BackwardKolmolgorovEquation_Fianl} up to terms of order $O(\Delta t^{k})$.
\end{remark}
\subsection{First order modified equation}\label{sec:FirstOrder_ModifiedEq}
\noindent
In this section, we shall analyze the numerical errors obtained by our symplectic splitting scheme and Euler Maruyama scheme \cite{Platen:1992}, respectively. We find that the solution obtained by the symplectic splitting scheme follows an asymptotic Hamiltonian while the solution obtained by the Euler Maruyama scheme does not. With our new method, we can achieve a linear growth (instead of an exponential growth) of the global error when we compute effective diffusivity.

After numerical discretization, we  find the following expansion using a first order weak method at $t=\Delta t$,
\begin{align}\label{FirstOrderModified_GeneralErrorForm}
\phi^{num}(x,\Delta t)= \phi_0(x) + \Delta t \mathcal{L}_0 \phi_0(x) + \Delta t^2 \mathcal{A}_1  \phi_0(x) +  O(\Delta t^{3}),
\end{align}
where $\mathcal{A}_1$ is a partial differential operator acting on $\phi_0(x)$ that depends on the choice of the numerical method used to solve Eq.\eqref{eqn:simplifiedSDE}. If we choose a convergent method to discretize the operator $\mathcal{L}_0$ in Eq.\eqref{FirstOrderModified_GeneralErrorForm} and Eq.\eqref{BackwardKolmolgorovEquation_OneStepWithTaylor}, then
the local truncation error is $O(\Delta t^2)$ and the numerical scheme is of weak order one. We refer to \cite{Platen:1992} for the definition and discussion of the weak convergence and strong convergence.

In detail, let $X^{num}(\Delta t)=(p(\Delta t),q(\Delta t))$ denote the numerical solution obtained by one specific choice of the numerical method
in solving Eq.\eqref{eqn:simplifiedSDE}. For instance, if we choose the symplectic splitting method stated in Eq.\eqref{NewStochasticIntegrators_Scheme_L2}, we get
\begin{align}\label{eqn:symsp1}
\begin{cases}
p(\Delta t)=p_0-\Delta tg(\frac{\Delta t}{2},q_0)+\sigma \Delta W_1,\\
q(\Delta t)=q_0+\Delta tf(\frac{\Delta t}{2},p_0-\Delta tg(\frac{\Delta t}{2},q_0))+\sigma \Delta W_2.
\end{cases}
\end{align}
Now $\Delta W_1$, $\Delta W_2$ are two independent random variables of the form $\sqrt{\Delta t}\mathcal{N}(0,1)$. To get $\mathcal{A}_1$, we only need to expand $E(\phi_0(p(\Delta t),q(\Delta t)))$ around point $\phi_0(p_0,q_0)$ along the time variable $\Delta t$. Since we are dealing with a separable Hamiltonian $H$, the operator splitting scheme helps us
obtain a straight-forward adaptive interpolation of Eq.\eqref{eqn:symsp1} for $t\in[0,\Delta t]$, saying $X_t^{num}$. We then have the form\cite{Zygalakis:2011},
\begin{align}
\phi^{num}(x,t)&=E[\phi_0(X_t^{num})|X_0=(p_0,q_0)]\\&=\phi_0(x) + \Delta t \mathcal{L}_0 \phi_0(x) + \Delta t^2 \mathcal{A}_1  \phi_0(x) +  O(\Delta t^{3})\label{eqn:direct_expansion1}
\end{align}
  In the BEA \cite{Reich:99}, we aim to find the generator $\mathcal{L}^{num}$ of this process and the associated backward Kolmogorov equation,
\begin{align}\label{Modified_FlowOperator}
\begin{cases}
\frac{\partial}{\partial t}\phi^{num}=\mathcal{L}^{num}\phi^{num}\\
\phi^{num}(x,0)=\phi_0(x).
\end{cases}
\end{align}
We now denote the generator of this modified equation in an asymptotic form in terms of $\Delta t$,
\begin{align}\label{generator_modifiedEq}
\mathcal{L}^{num}\equiv\mathcal{L}_0 +  \Delta t \mathcal{L}_1 +  \Delta t^2 \mathcal{L}_2 + \cdot\cdot\cdot.
\end{align}
Recall that the operator $\mathcal{L}_0$ is defined in Eq.\eqref{HamiltonianFlowOperator} and the definition of operators $\mathcal{L}_i$, $i\geq 1$ depends on the choice of the numerical method in solving Eq.\eqref{eqn:simplifiedSDE}, i.e. sub \eqref{generator_modifiedEq} into \eqref{BackwardKolmolgorovEquation_Fianl} then compare with \eqref{eqn:direct_expansion1}, we get
\begin{align}
\mathcal{L}_1=\mathcal{A}_1-\frac{1}{2}\mathcal{L}_0^2.
\end{align}
Now let us denote the truncated generator by,
\begin{align}
\mathcal{L}^{\Delta t,k}:=\mathcal{L}_0+\Delta t\mathcal{L}_1+\cdots+\Delta t^k\mathcal{L}_k.
\end{align}
and denote the corresponding modified flow (if it exists),
\begin{align}\label{eqn:truncatedmodifiedflow}
\begin{cases}
\frac{\partial}{\partial t}\phi^{\Delta t}=\mathcal{L}^{\Delta t,k}\phi^{\Delta t}\\
\phi^{\Delta t}(x,0)=\phi_0(x).
\end{cases}
\end{align}
Inspired by the weak convergence proof in \cite{Platen:1992}, we shall focus on estimating the
upper bound of the uniform numerical error for the perturbed flows.
\begin{lemma}\label{lem:FlowComparing}
	Let $\phi^{num}$ and $\phi^{\Delta t}$ be defined in \eqref{Modified_FlowOperator} and \eqref{eqn:truncatedmodifiedflow}, respectively. We assume that $\phi_0\in \mathcal{C}^{\infty}$ and its Ito-Taylor expansion coefficients in the hierarchy set $\Gamma_{k+1}\bigcup B(\Gamma_{k+1})$ are Lipschitz and have at most linear growth. If the solution to the first order modified flow, $\phi^{\Delta t}$ converges to $\phi$ as $\Delta t\to 0$, then we have the following error estimate
	\begin{align}
	||\phi^{num}(x,t)-\phi^{\Delta t}(x,t)||\le C(T)\Delta t^{k+1}
	\end{align}
\end{lemma}

\begin{proof}
	Eq.\eqref{FirstOrderModified_GeneralErrorForm} shows that the operator $L^{\Delta t}$ approximates
the operator $L^{\Delta t,k}$ locally in the time interval $[0,\Delta t]$ with the truncation error $O(\Delta t^{k+2})$. 
This implies that $X_t^{\Delta t}$ is a $k+1$-th order weak approximation to the SDE related to $X^{\Delta t,k}_t$ locally, i.e.
\begin{equation}
\phi_0(X^{num}_{\Delta t})-\phi_0(X^{\Delta t,k}_{\Delta t})=\phi_0(X_{0}^{num})-\phi_0(X_{0}^{\Delta t,k})+\sum_{\alpha\in B(\Gamma_{k+1})}I_{\alpha}[\phi_{0,\alpha}(X^{\Delta t,k}_{(\cdot)})]_{0,\Delta t}
\end{equation}
Here we refer to the Chapter 5.5 in \cite{Platen:1992} for more detailed definition of multi-index stochastic Ito integration notation $I_{\alpha}$. Proposition 5.11.1 in \cite{Platen:1992} gives an estimate for the $I_{\alpha}$,
\begin{equation}
|E\sum_{\alpha\in B(\Gamma_{k+1})}I_{\alpha}[\phi_{0,\alpha}(X^{\Delta t,k}_{(\cdot)})]_{0,\Delta t}|\leq C(L^{\Delta t,k})\Delta t^{k+2}
\end{equation}
Since the operator $L^{\Delta t,k}$ approximates $L_0$, $\lim_{\Delta t\to 0}C(L^{\Delta t,k})=C(L)$. Combining with Lipschitz and linear growth condition, the final weak convergence order should be  $C(T)O(\Delta t^{k+1})$ when $\Delta t$ is small enough.
\end{proof}
\begin{remark}
Figure \ref{fig:workonflow} shows the general procedure of our convergence analysis.
Our goal is to develop efficient numerical method so that we can reduce the error in calculating effective diffusivity $|D^{E,num}-D^E|$, which is the dashed line on the left. Terms (namely $D^{E,\Delta t}$, $X^{\Delta t,k}_t(or\ X^{\Delta t}_t)$, $\mathcal{L}^{\Delta t,k}:\phi^{\Delta t}$) are introduced from the BEA and play intermediate roles between the numerical solutions (shown in the upper row) and the analytic ones (shown in the bottom row). This framework clearly reveals the main sources of error (i.e. $|D^{E,\Delta t}-D^E|$). These notations are commonly used in this paper.
	\begin{figure}[htbp]
		\centering
		\begin{tikzpicture}
		\matrix (m) [matrix of math nodes,row sep=6em,column sep=8em,minimum width=2em]
		{	D^{E,num}&X^{num}_t & \mathcal{L}^{num}:\phi^{num} \\
			D^{E,\Delta t}&X^{\Delta t,k}_t(or\ X^{\Delta t}_t) &\mathcal{L}^{\Delta t,k}:\phi^{\Delta t} \\
			D^{E}&X_t& \\};
		\path[-stealth]
		(m-1-2) edge node [above,align=left,font=\small] {Weak Taylor Exp} (m-1-3)
		(m-1-2) edge node [below] {at $t=\Delta t$} (m-1-3)
		(m-2-1) edge node [right,align=left,font=\small] {Comparing in cell\\T-invariant, $O(\Delta t^{k})$\\(Thm.\ref{thm:ErrorEstimateOfCellSolution})} (m-3-1)
		(m-2-3) edge node [above] {$\exists$ path} (m-2-2)
		(m-1-3) edge node [left,align=right] {Truncated operator\\$C(T)O(\Delta t^{k+1})$\\(Lem.\ref{lem:FlowComparing})} (m-2-3)
		(m-1-1) edge [bend right, dashed,align=right] node [left] {Final Error\\$O(\Delta t^k)$}(m-3-1)
		(m-1-2) edge node [above,align=left,font=\small] {Calculate from\\ Monte-Carlo path} (m-1-1)
		(m-2-2) edge node [above,align=left,font=\small] {Homogenization \\Approach} (m-2-1)
		(m-3-2) edge node [above,align=left,font=\small] {Particle Definition} (m-3-1)
		(m-3-2) edge node [below,align=left,font=\small] {Homogenization Approach} (m-3-1)
		(m-3-2) edge [-,dashed,align=left] node [right] {Both Hamiltonian flow\\(Thm.\ref{thm:FollowAsymptoticHamiltonian})} (m-2-2);
		\end{tikzpicture}
		\caption{Illustration of backward error analysis for $k$-th order weak scheme}
		\label{fig:workonflow}
	\end{figure}
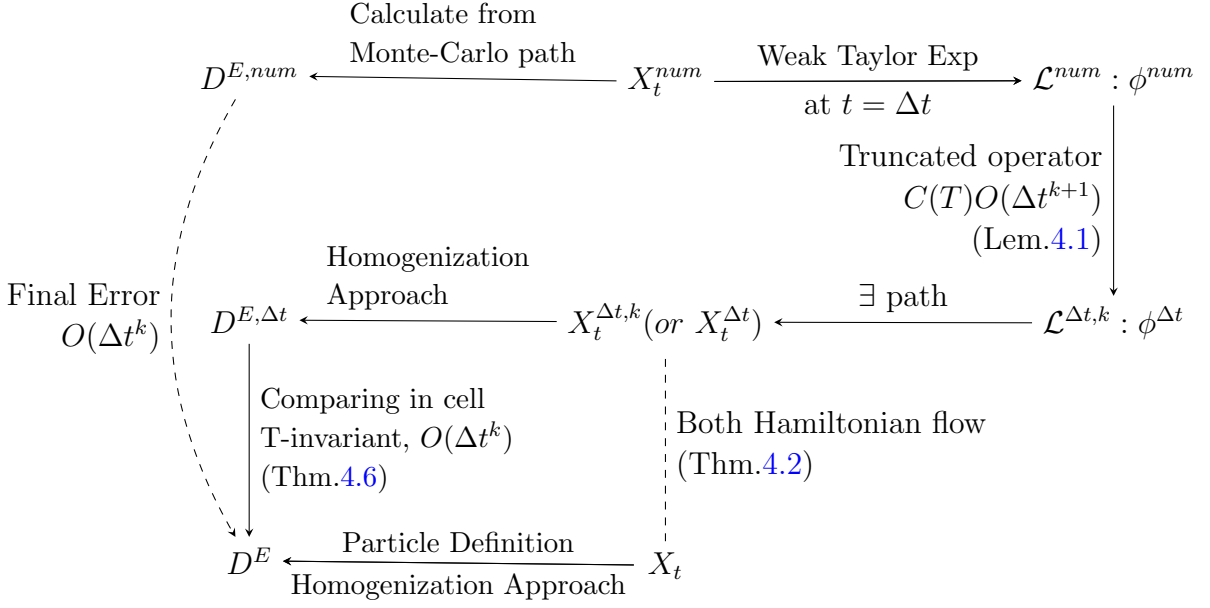
	\end{remark}
The foregoing derivation shows that modified flows allow us to approximate the interpolation of numerical solution with a higher-order accuracy. Hence the modified flows dominate the error in numerical result.
Now we intend to study the behavior of the modified flows.
\begin{theorem}\label{thm:FollowAsymptoticHamiltonian}
For the stochastic differential equation system Eq.\eqref{eqn:simplifiedSDE} with a time independent  and separable Hamiltonian $H(p,q)$ Eq.\eqref{SeparableHamiltonian}, the numerical solution obtained using the symplectic splitting scheme follows an asymptotic Hamiltonian $H^{\Delta t}(p,q)$, or equivalently, the first order modified equation (density function) of the solution is divergence-free. The invariant measure on torus  (defined by $\mathbb{R}^d/\mathbb{Z}^d$, when period is $1$) remains trivial. While the numerical solution obtained using the Euler Maruyama scheme does not have these properties.
\end{theorem}
\begin{proof}
We shall compare the generators of modified equations obtained by using the symplectic splitting scheme and Euler Maruyama scheme, respectively. More specifically, we compare the operator $\mathcal{L}_1$ in Eq.\eqref{generator_modifiedEq} obtained from different methods.
In the symplectic splitting scheme, we compute the weak Taylor expansion at time $t=\Delta t$ and get,
\begin{align}\label{generator_modifiedEq_L1symflow}
	\mathcal{L}_1\phi=(\mathcal{A}_1-\frac{1}{2}\mathcal{L}_0^2)\phi=(\frac{1}{2}fg'+\frac{\sigma^2}{4}g'')\phi_p+
(-\frac{1}{2}f'g-\frac{\sigma^2}{4}f'')\phi_q+(-\frac{\sigma^2}{2}f'+\frac{\sigma^2}{2}g')\phi_{pq}
\end{align}
Hence, the modified flow of $X^{\Delta t,k}$ can be written as
\begin{align}\label{symflow}
		\begin{cases}
			dp=(-g+(\frac{1}{2}fg'+\frac{\sigma^2}{4}g'')\Delta t)dt+\sigma dW_1+\Delta t\frac{\sigma}{2} g'dW_2\\
			dq=(f-(\frac{1}{2}f'g+\frac{\sigma^2}{4}f'')\Delta t)dt+\sigma dW_2-\Delta t\frac{\sigma}{2} f'dW_1
		\end{cases}
\end{align}
Similarly, in the Euler Maruyama scheme, we get that
\begin{align}\label{generator_modifiedEq_L1Eulerflow}
\mathcal{L}_1\phi=(\mathcal{A}_1-\frac{1}{2}\mathcal{L}_0^2)\phi=(\frac{1}{2}fg'+\frac{\sigma^2}{4}g'')\phi_p+(\frac{1}{2}f'g-
\frac{\sigma^2}{4}f'')\phi_q+(-\frac{\sigma^2}{2}f'+\frac{\sigma^2}{2}g')\phi_{pq}
\end{align}
And the associated modified flow can be written as
\begin{align}\label{eulerflow}
\begin{cases}
dp=(-g+(\frac{1}{2}fg'+\frac{\sigma^2}{4}g'')\Delta t)dt+\sigma dW_1+\Delta t\frac{\sigma}{2} g'dW_2\\
dq=(f-(-\frac{1}{2}f'g+\frac{\sigma^2}{4}f'')\Delta t)dt+\sigma dW_2-\Delta t\frac{\sigma}{2} f'dW_1
\end{cases}
\end{align}
Compare the results from Eq.\eqref{symflow} and Eq.\eqref{eulerflow}, we can easily find that Eq.\eqref{symflow} follows an asymptotic Hamiltonian,
\begin{align}\label{symH}
H^{\Delta t}\equiv H-\Delta t\big(\frac{1}{2}fg+\frac{\sigma^2}{4}(f'+g')\big),
\end{align}
while the flow Eq.\eqref{eulerflow} obtained from the Euler Maruyama scheme does not have this structure. Furthermore, we introduce notations $v_1$ and $d_1$ to denote extra terms in the
modified flow Eq.\eqref{symflow}, which are defined as
\begin{equation}\label{extraterm_modifiedflow}
v_1=\left(\begin{array}{cc}
\frac{1}{2}fg'+\frac{\sigma^2}{4}g''\\-\frac{1}{2}f'g-\frac{\sigma^2}{4}f'
\end{array}\right), \quad \text{and} \quad
d_1=\left(\begin{array}{cc}
0 &\frac{1}{2}g'\\-\frac{1}{2}f' &0
\end{array}\right)
\end{equation}
Since our numerical method solves a stochastic differential equations determined by
a modified flow Eq.\eqref{symflow}, the density function of particles $u(x,t)$ obtained from our method satisfy a modified
Fokker-Planck equation given by
\begin{equation}\label{ModifiedFokkerPlanck-eq}
u_t=-(v+\Delta tv_1)\nabla u+D_0\nabla\nabla:(I+\Delta t D_1)u,
\end{equation}
where $D_1=\big((I_d+\Delta td_1)(I_d+\Delta td_1)^T-T_d\big)/\Delta t=\left(\begin{array}{cc}
\frac{\Delta t}{4}(g')^2 &\frac{1}{2}(g'-f')\\\frac{1}{2}(g'-f') &\frac{\Delta t}{4}(f')^2
\end{array}\right)$ and we have used the condition $\nabla\cdot v_1=0$ to get $\nabla \big((v+\Delta t v_1)u\big)=(v+\Delta tv_1)\nabla u$.  The inner product
between matrices is denoted by $A:B = tr(A^{T}B) = \sum_{i,j} a_{ij}b_{ij}$.  It follows
that $\Delta  = \nabla\nabla:I $ and $\nabla\nabla:D_1$ are defined accordingly. Then we can check that Eq.\eqref{ModifiedFokkerPlanck-eq} admits trivial invariant measure $u_0(x,t)\equiv 1$.
\end{proof}
We can repeat a similar calculation and generalize the result in Theorem \ref{thm:FollowAsymptoticHamiltonian} to a general time dependant and separable Hamiltonian. Therefore, we obtain the result as follows,
\begin{corollary}\label{col:FollowAsymptoticHamiltonian}
For the stochastic differential equation system Eq.\eqref{eqn:simplifiedSDE} with a time dependent
and separable Hamiltonian $H$ Eq.\eqref{SeparableHamiltonian}, the numerical solution obtained using the
symplectic splitting scheme follows an asymptotic Hamiltonian $H^{\Delta t}$, or equivalently, the first order modified equation (density function) of the solution is divergence-free. The invariant measure on torus (defined by $\mathbb{R}^d/\mathbb{Z}^d$, when period is $1$) remains trivial.. While the numerical solution obtained using the Euler Maruyama scheme does not have these properties.
\end{corollary}
\begin{proof}
We repeat the same computation as we did in Thm.\ref{thm:FollowAsymptoticHamiltonian}. In the
symplectic splitting scheme, we find that the corresponding modified flow can be written as
	\begin{align}\label{symflow_timedependent}
	\begin{cases}
	dp=\big(-g+(\frac{1}{2}fg'+\frac{\sigma^2}{4}g''+\frac{1}{2}g_t)\Delta t\big)dt+\sigma dW_1+\Delta t \frac{\sigma}{2} g'dW_2\\
	dq=\big(f-(\frac{1}{2}f'g+\frac{\sigma^2}{4}f''+\frac{1}{2}f_t)\Delta t\big)dt+\sigma dW_2-\Delta t \frac{\sigma}{2} f'dW_1
	\end{cases}
	\end{align}
The rest part is similar with Thm.\ref{thm:FollowAsymptoticHamiltonian}.
\end{proof}
Before we end this subsection, we use an example to demonstrate our main idea. We consider the flow driven by the Taylor-Green velocity field,
\begin{equation}\label{Taylor_Green}
	\begin{cases}
	dp=-\cos(q)\sin(p)dt+\sigma dW_1,\\
	dq=\sin(q)\cos(p)dt+\sigma dW_2.
	\end{cases}
\end{equation}
By introducing two variables $P=p+q$ and $Q=p-q$, we know the dynamic system Eq.\eqref{Taylor_Green} possesses a separable
Hamiltonian, $H=-\cos P-\cos Q$ and the system can be expressed by
\begin{equation}
\begin{cases}
dP=-\sin Q+\sqrt{2}\sigma d\eta_1,\\
dQ=\sin P+\sqrt{2}\sigma d\eta_2,
\end{cases}
\end{equation}
where $\eta_1$ and $\eta_2$ are two independent Brownian motions that are linear combinations of $W_1$ and $W_2$. Substituting into Eq.\eqref{symH} and Eq.\eqref{symflow}, we get,
\begin{align}
H^{\Delta t}=H-\Delta t\big(\frac{1}{2}\sin P\sin Q+\frac{\sigma^2}{2}(\cos P+\cos Q)\big),
\end{align}
and
\begin{equation}
\begin{cases}
dP=-\frac{\partial H^{\Delta t}}{\partial Q}dt+\sqrt{2}\sigma d\eta_1+\Delta t\frac{\sigma}{\sqrt{2}}\cos Qd\eta_2,\\
dQ=\frac{\partial H^{\Delta t}}{\partial P}dt+\sqrt{2}\sigma d\eta_2+\Delta t\frac{\sigma}{\sqrt{2}}\cos Pd\eta_1.
\end{cases}
\end{equation}

Up to now, the new integrator Eq.\eqref{eqn:HamiltonianSDE} is shown to preserve structure of original Hamiltonian system Eq.\eqref{eqn:simplifiedSDE} asymptotically at $O(\Delta t)$. In next subsection, we study effective diffusivity as a behavior of the structure.
\subsection{Error analysis for computing the effective diffusivity}\label{sec:EstiResidualDiffusivity}
\noindent
Recalling Eq.\eqref{Def_EffectiveDiffusivity_Lagrangian}, only distribution of the process is needed, so Eulerian framework is sufficient to get an error estimate. For sake of comparison,  we re-write the effective diffusivity formula Eq.\eqref{Def_EffectiveDiffusivity_Euler} for Eq.\eqref{eqn:HamiltonianSDE} as,
\begin{align}\label{EffectiveDiffusivityTensor}
D^E=D_0\langle(I_d+\nabla w)(I_d+\nabla w)^T\rangle_p.
\end{align}
where $D_0=\sigma^2/2$, which is globally used in context, and cell problem $w$ satisfies,
\begin{equation}\label{FokkerPlanck-cellproblem}
w_t+(v\cdot\nabla w)-D_0\Delta w=-v.
\end{equation}, with the velocity filed $v=(-g, f)^{T}$.
To study effective diffusivity in Eq.\eqref{ModifiedFokkerPlanck-eq}, we turn to the Section 3.10 of \cite{Bensoussan:78}, where
an exact formula for $D^E$ in a non-constant diffusion case is provided. Let $w^{\Delta t}\equiv w^{\Delta t}(t,x)$ denote the periodic solution of the cell problem that is corresponding
to the modified Fokker-Planck equation Eq.\eqref{ModifiedFokkerPlanck-eq}, i.e., $w^{\Delta t}$ satisfies the
following equation
\begin{equation}\label{ModifiedFokkerPlanck-cellproblem}
w^{\Delta t}_t=-(v+\Delta tv_1)\cdot\nabla w^{\Delta t}+D_0\nabla\nabla:(I+\Delta t D_1)w^{\Delta t}-(v+\Delta tv_1).
\end{equation}
 We introduce the operators  $\mathcal{P}_0w^{\Delta t} \equiv - v \nabla w^{\Delta t} + \frac{\sigma^2}{2} \Delta w^{\Delta t}$ and
 $\mathcal{P}_1w^{\Delta t} \equiv - v_1 \nabla w^{\Delta t} + \frac{\sigma^2}{2}\nabla\nabla: D_1 w^{\Delta t}$ to simplify Eq.\eqref{ModifiedFokkerPlanck-cellproblem} as
\begin{equation}\label{eqn:modified_cell}
w^{\Delta t}_t=(\mathcal{P}_0+\Delta t\mathcal{P}_1)w^{\Delta t}-(v+\Delta tv_1).
\end{equation} Now by Theorem \ref{thm:FollowAsymptoticHamiltonian} and Corollary \ref{col:FollowAsymptoticHamiltonian}, Eq.\eqref{symflow} admits trivial invariant measure, so formula for the effective diffusivity tensor turns into,
\begin{equation}\label{ModifiedEffectiveDiffusivityTensor}
D^{E,\Delta t}=D_0\Big\langle (I_d+\nabla w^{\Delta t})(I_d+\Delta t D_1) (I_d+\nabla w^{\Delta t})^T \Big \rangle_p.
\end{equation}
The modified cell problem \eqref{ModifiedFokkerPlanck-cellproblem} and the corresponding effective diffusivity tensor Eq.\eqref{ModifiedEffectiveDiffusivityTensor} enable us to analyse the error in our new method.
\begin{lemma}\label{thm:ExistenceOfCellSolution}
 Eq.\eqref{ModifiedFokkerPlanck-cellproblem}  has a unique solution if the condition $\int_{U_T}w^{\Delta t}dxdt=0$ holds, where $U_T=[0,T]\times U$ is the space-time domain for the periodic function $w$.
\end{lemma}
\begin{proof}
 We first notice that when $\Delta t\ll D_0$, the operator $(\mathcal{P}_0+\Delta t\mathcal{P}_1)$ is uniformly elliptic. The space average of the source  term $-(v+\Delta tv_1)$ vanishes. By the Fredholm alternative, Eq.\eqref{eqn:modified_cell} has nontrivial solutions if $-(v+\Delta tv_1)\not\equiv 0$. Then, using the maximum principle, we get the conclusion that the solution $w^{\Delta t}$ to Eq.\eqref{ModifiedFokkerPlanck-cellproblem} is unique if the condition $\int_{U_T}w^{\Delta t}dxdt=0$ is satisfied.
\end{proof}
Now we derive regularity estimate in this Poincar� map problem \eqref{ModifiedFokkerPlanck-cellproblem}.
\begin{theorem}\label{thm:RegularityOfCellSolution}
	Suppose $w=w(t,x)$ is a space-time periodic solution over the domain $U_T=[0,T]\times U$, which satisfies
	\begin{equation}\label{eqn:regularity est}
	w_t+(v\cdot \nabla w)-D:\nabla\nabla w=S,  \ (t,x)\in U_T=[0,T]\times U,
	\end{equation}
	where $\nabla\cdot v=0$, $D$ is symmetric and its eigen values are between $[D_-,D_+],\ \forall(x,t)$ ,  $S=S(t,x)$ is the source term, which vanishes in average at any time $t$.
	Then, we have the regularity estimate for $w$ as $|\nabla w|_{L_2(U_T)}\leq C|S|_{L_2(U_T)}$ where the constant
	$C$ depends only on the length of the physical domain $U$ and the parameter $D$.
\end{theorem}
\begin{proof}
	We multiply the equation Eq.\eqref{eqn:regularity est} by $w^T$ and integrate in $U$
	\begin{equation}\label{eqn:regularity est2}
	\int_U (w^T w_t+w^Tv\nabla w- w^TD:\nabla\nabla w)dx=\int_U w^T Sdx
	\end{equation}
	We shall notice that,
	\begin{align*}
	\int_U w^T w_t dx  &= \frac{d}{dt}\int_U |w|^2 dx, \\
	\int_U w^Tv\nabla w dx &= -\int_Uw^Tv\nabla w dx=0, \\
	\int_U-w^TD:\nabla\nabla w dx &= \int_U \nabla w^T D\nabla wdx,
	\end{align*}
	where we have used the condition $\nabla\cdot v=0$.  Then, we integrate Eq.\eqref{eqn:regularity est2} over the time period $[0,T]$ and
	the periodic condition of $w$ implies
	\begin{equation}
	\int_{U^T} \nabla w^T D\nabla wdx=\int_{U^T}w^TSdxdt
	\end{equation}
	Let $\bar{w}(t)$ denote the space average of $w$ at time $t$. Since $S$ vanishes in space average at any time $t$, we have
	\begin{equation}
	\int_{U^T}\bar{w}^TSdxdt=0.
	\end{equation}
	In addition, we get the equality
	\begin{equation}
	\big(\int_{U^T} \nabla w^T D\nabla wdx\big)^2=\big(\int_{U^T}(w^T-\bar{w}^T)Sdxdt\big)^2
	\end{equation}
	Applying Poincare inequality on the right hand side and Cauchy-Schwartz on the left, we obtain the estimate
	\begin{equation}\label{estimate_Poincare}
	\int_{U^T} \nabla w^T D\nabla wdx
	\geq D_-\int_{U^T}|\nabla w|^2dxdt\geq\int_{[0,T]}C_U\int_U |w-\bar{w}|^2dxdt=\int_{U^T} |w-\bar{w}|^2dxdt
	\end{equation}
	\begin{equation}\label{estimate_CauchySwartz}
	(\int_{U^T}(w^T-\bar{w}^T)Sdxdt)^2\leq \int_{U^T}|S|^2dxdt \int_{U^T}|w-\bar{w}|^2dxdt
	\end{equation}
	Combining the inequalities Eq.\eqref{estimate_Poincare} and Eq.\eqref{estimate_CauchySwartz}, we finally get the regularity estimate in $L_2\ norm$.
	\begin{equation}
	|\nabla w|_{L_2(U^T)}\leq \frac{C(U)}{D_-}|S|_{L_2(U^T)}.
	\end{equation}
\end{proof}
Given the regularity estimate of $w^{\Delta t}$ in \eqref{ModifiedFokkerPlanck-cellproblem}, we can easily get estimate for the error
between solutions to Eq.\eqref{FokkerPlanck-cellproblem} and Eq.\eqref{ModifiedFokkerPlanck-cellproblem}.
We summarize the main result into the following theorem.
\begin{theorem}\label{thm:ErrorEstimateOfCellSolution}
Let $w(x,t)$ and $w^{\Delta t}(x,t)$ be the solution to the Eq.\eqref{FokkerPlanck-cellproblem} and Eq.\eqref{ModifiedFokkerPlanck-cellproblem}, respectively. We have the estimate
$|\nabla w - \nabla w^{\Delta t}|_{L_2(U_T)}\leq C_U\frac{\Delta t}{D_0}|S_e|_{L_2(U_T)}$, where $S_e=\mathcal{P}_1w^{\Delta t}-v_1$ is the source term.
\end{theorem}
\begin{proof}
Let $e\equiv e(x,t)=w(x,t)-w^{\Delta t}(x,t)$ denote the error. One can easily find that $e$ is a space-time periodic function
over $U_T=[0,T]\times U$ and satisfies the following equation
\begin{equation}\label{eqn:ErrorEstimate}
e_t+(v\cdot \nabla e)-D_0\Delta e=(\Delta t)S_e,
\end{equation}
where the source term $S_e$ is defined above. So we directly apply the regularity estimate for the parabolic-type
equation obtained in Thm.\ref{thm:RegularityOfCellSolution} and obtain,
\begin{equation}\label{eqn:ErrorEstimateResult}
|\nabla e|_{L_2(U^T)}\leq C(U)\frac{\Delta t}{D_0}|\mathcal{P}_1w^{\Delta t}-v_1|_{L_2(U^T)}
\end{equation}
Again when $\Delta t\ll D_0$, the operator $\frac{\partial}{\partial t}+(\mathcal{P}_0+\Delta t\mathcal{P}_1)$ is uniformly parabolic and the diffusion coefficients $D=D_0+\Delta tD_1$ is positive and uniformly bounded below (i.e.$D_-\to D_0$) for any $\Delta t$ small enough.
By regularity estimate of parabolic equation (a concrete estimate may comes from \cite{Evans:2010}), we can get $w^{\Delta t}$, $\nabla w^{\Delta t}$ and $\nabla \nabla :w^{\Delta t}$ are uniformly bounded in $L_2(U_T)$ for any $\Delta t$ small enough, hence,
\begin{align}
|\mathcal{P}_1w^{\Delta t}-v_1|_{L_2(U_T)}=|(-v_1\nabla +D_0\nabla\nabla:D_1)w^{\Delta t}-v_1|_{L_2(U_T)}\leq C,
\end{align}
where the constant $C$ is independent of $\Delta t$.
\end{proof}
Finally, based on the error estimate for the solutions to the cell problems \eqref{FokkerPlanck-cellproblem} and \eqref{ModifiedFokkerPlanck-cellproblem}, we are able to get the error analysis for the effective diffusivity.
\begin{remark}
	From Eq.\eqref{eqn:ErrorEstimateResult}, we shall state, a proper setting in calculating effective diffusivity should be
	\begin{equation}\label{eqn:timestep_selection}
		\Delta t\sim D_0=\frac{\sigma^2}{2}.
	\end{equation}
	\end{remark}
\begin{corollary}\label{lem:EffectiveDiffusivityError}
Let $D^E$ and $D^{E,\Delta t}$ denote the effective diffusivity tensor computed by Eq.\eqref{EffectiveDiffusivityTensor} and Eq.\eqref{ModifiedEffectiveDiffusivityTensor}. Then, the error of the effective diffusivity tensor can be bounded by
\begin{align}\label{eqn:EffectiveDiffusivityError}
|D^{E,\Delta t}-D^E|\le C \Delta t,
\end{align}
where the constant $C$ does not depend on time $T$.
\end{corollary}
\begin{proof}
Recalling Eq.\eqref{ModifiedEffectiveDiffusivityTensor}, $D^{E,\Delta t}=D_0\Big\langle (I_d+\nabla w^{\Delta t})(I_d+\Delta t D_1) (I_d+\nabla w^{\Delta t})^T \Big \rangle_p$ where
$D_1=\left(\begin{array}{cc}
\frac{\Delta t}{4}(g')^2 &\frac{1}{2}(g'-f')\\\frac{1}{2}(g'-f') &\frac{\Delta t}{4}(f')^2
\end{array}\right)$. We shall see the fact that $\langle \frac{1}{2}(g'-f')\rangle_p=0$, $\langle \nabla w^{\Delta t}\rangle_p=0$. Hence,
\begin{align}
D^{E,\Delta t}-D^E=&D_0(\langle \nabla w^{\Delta t}\nabla w^{\Delta t,T}-\nabla w\nabla w^T\rangle_p+O(\Delta t^2))\\
=&D_0\big((\nabla w^{\Delta t}-\nabla w)\nabla w^{T}+\nabla w(\nabla w^{\Delta t}-\nabla w)^{T}\\&+(\nabla w^{\Delta t}-\nabla w)(\nabla w^{\Delta t}-\nabla w)^{T}+O(\Delta t^2)\big).
\end{align}
Then considering Thm.\ref{thm:RegularityOfCellSolution}, and  we can find that the order of the error in Eq.\eqref{eqn:EffectiveDiffusivityError} is $O(\Delta t)$.
\end{proof}
\begin{theorem}\label{thm:FinalEstimate}
 Solution of Eq.\eqref{eqn:simplifiedSDE} is denoted as $X_t$ and adaptive interpolated process of Eq.\eqref{NewStochasticIntegrators_Scheme_L1} as $X^{num}_t$. To calculate effective diffusivity of  $X^{num}_t$ which both start at $x$, we  define $\tilde{D}^{E,num}(x,t)=E[\frac{(X^{num}_t-X_0)\otimes(X^{num}_t-X_0)}{2t}|X_0=x]$ for $0< t\le T$.
	\begin{align}\label{eqn:dethm}
	\sup_{x}|\tilde{D}^{E,num}(x,t)-D^E|\le C\Delta t+C(T)\Delta t^2
	\end{align}
\end{theorem}
\begin{proof}
 	Let  $\tilde{D}^{E,\Delta t}(x,t)=E[\frac{(X_t^{\Delta t}-X_0^{\Delta t})\otimes(X_t^{\Delta t}-X^{\Delta t}_0)}{2t}|X_0^{\Delta t}=x]$. For any $\epsilon>0$. We assume $\phi_0(x)=\sqrt{\epsilon+{(x-X_0)^T(x-X_0)}}$ in Lem.\ref{lem:FlowComparing}, we see that $|\sqrt{\tilde{D}^{E,num}(x,t)}-\sqrt{\tilde{D}^{E,\Delta t}(x,t)}|\le C(T)\Delta t^2$. By homogenization theory (like \cite{Bensoussan:78}, \cite{Stuart:08}), we shall see $\lim_{t\rightarrow\infty}|\tilde{D}^{E,\Delta t}(x,t)-D^{E,\Delta t}|=0$. Finally, Col.\ref{lem:EffectiveDiffusivityError} states $|D^{E,\Delta t}-D^E|\le C \Delta t^2$.  Eq.\eqref{eqn:dethm} is the result of triangle inequality.
\end{proof}
\begin{remark}
	We shall see that in calculating effective diffusivity, we approximate $D^E$ by $\tilde{D}^{E,num}$ in which taking expectation corresponds to simulation ignoring error of Monte-Carlo.
\end{remark}
\begin{remark}
If a long-time behavior of a flow (i.e. effective diffusivity) can be approximated by a truncated flow of the numerical method, the error in approximating such behavior may be dominant by the truncated flow which can be studied analytically. In case of Thm.\ref{thm:FinalEstimate}, general error analysis (like in \cite{Platen:1992}) will state $|\tilde{D}^{E,num}(x,t)-D^E|\le C(T)\Delta t$ where $C(T)$ grows exponentially as $T\to\infty$.
\end{remark}

\section{Numerical results}\label{sec:Numerical_example}
\noindent
In this section, we shall apply our methods to investigate the behaviors of several time-dependent chaotic and stochastic flows.
We are interested in understanding the mechanisms of the diffusion enhancement, the existence of residual diffusivity, highlighting the influence of Lagrangian chaos on flow transport, and long-time performance of different numerical methods.
\subsection{Chaotic cellular flow with oscillating vortices}
\noindent
For the first example, we consider the passive tracer model in which the velocity field is given by a chaotic cellular flow with oscillating vortices. Specifically, the flow is generated by a Hamiltonian defined as $H(t,p,q)=-\frac{1}{k}\cos(kp+B\sin(\omega t))\sin(kq)$. The motion of a particle moving in this chaotic cellular flow is described by the SDE,
\begin{equation}\label{eqn:CellularFlowOscillatingVortices}
\begin{cases}
dp=\sin(kp+B\sin(\omega t))\cos(kq)dt+\sigma dW_1,\\
dq=-\cos(kp+B\sin(\omega t))\sin(kq)dt+\sigma dW_2,
\end{cases}
\end{equation}
with initial data $(p_0,q_0)$. The behavior of Eq.\eqref{eqn:CellularFlowOscillatingVortices} with $\sigma=0$ was intensively studied in \cite{Moura:2010}, which is a two-dimensional incompressible flow representing a lattice of oscillating vortices or roll cells. Moreover, when $B=0$  the flow in Eq.\eqref{eqn:CellularFlowOscillatingVortices} turns into the classic Taylor-Green velocity field. In this setting real fluid elements follow trajectories that are level curves of its Hamiltonian. When $B\neq0$, the trajectories of the passive tracers differ from the streamlines, due to the oscillating vortex in the flow.

When $\sigma>0$ the dynamics of the Eq.\eqref{eqn:CellularFlowOscillatingVortices} will exhibit more structures, which is an interesting model problem to test the performance of our method.  We point out that when $B\neq0$ and $\sigma>0$, the long-time large-scale behavior of the particle model of Eq.\eqref{eqn:CellularFlowOscillatingVortices} has been studied by many researchers, for example in \cite{Fannjiang:94,StuartZygalakis:09}. It shows that the asymptotic behaviors of effective diffusivity $D^E \sim \sigma I_2$ (,or equivalently $D^E \sim\sqrt{2D_0}I_2$),
which means that  for this type of flow there does not exist residual diffusivity.

In our numerical experiments, we choose $k=2\pi$, $\omega=\pi$, $(p_0,q_0)=(0,0)$ in the SDE Eq.\eqref{eqn:CellularFlowOscillatingVortices}. The time step is $\Delta t=10^{-2}$ and the final computational time is $T=10^4$. We consider different $B$ to study the behaviours of effective diffusivity in vanishing viscosity (i.e. $\sigma \to 0$). We compare the numerical obtained using
the sympletic splitting scheme and Euler-Maruyama scheme. In our comparison, we use the same Monte Carlo samples to
discretize the Brownian motions $dW_1$ and $dW_2$. The sample number is $N_{mc}=5000$.

In Figure \ref{fig:eg1f1}, we show the numerical results of effective diffusivity $D^E_{11}$ obtained using different methods and parameters. Left part of the figure shows the results for Taylor-Green velocity field ($B=0$). One can see that the Euler-Maruyama scheme fails to achieve the theoretical analysis for $D^E$, i.e., $D^E \sim \sigma I_2$, while the result obtained using our sympletic splitting scheme agrees with the theory well. Right part of the figure shows the results for $B=2.72$. One again finds that the behaviors of the Euler-Maruyama scheme and our scheme are different.

To further compare the performance of the Euler-Maruyama scheme and our method, we repeat the same experiment with $k=2\pi$, $\omega=\pi$, $(p_0,q_0)=(0,0)$ and $\sigma=10^{-2}$ in Eq.\eqref{eqn:CellularFlowOscillatingVortices}, but try different time step $\Delta t$ with $B=0$ and $B=2.72$ correspondingly. In Figure \ref{fig:eg1f2}, we find that symplectic scheme can achieve very accurate results even using a relatively larger time step, while the Euler-Maruyama scheme cannot give the right answer even using a very smallest time step. As a result of our analysis \ref{lem:EffectiveDiffusivityError} and Eq.\eqref{eqn:CellularFlowOscillatingVortices}, we can say that the numerical result for $D_{11}^E$ has converged to the analytical result. Therefore, we conjecture that the time dependent cellular flow we studied in Eq.\eqref{eqn:CellularFlowOscillatingVortices} with $B=2.72$, we still have $D^E\sim\sigma I_2$. More theoretic analysis of this flow will be reported in our future work.
\begin{figure}[htbp]
	\centering
	\includegraphics[width=0.9\linewidth]{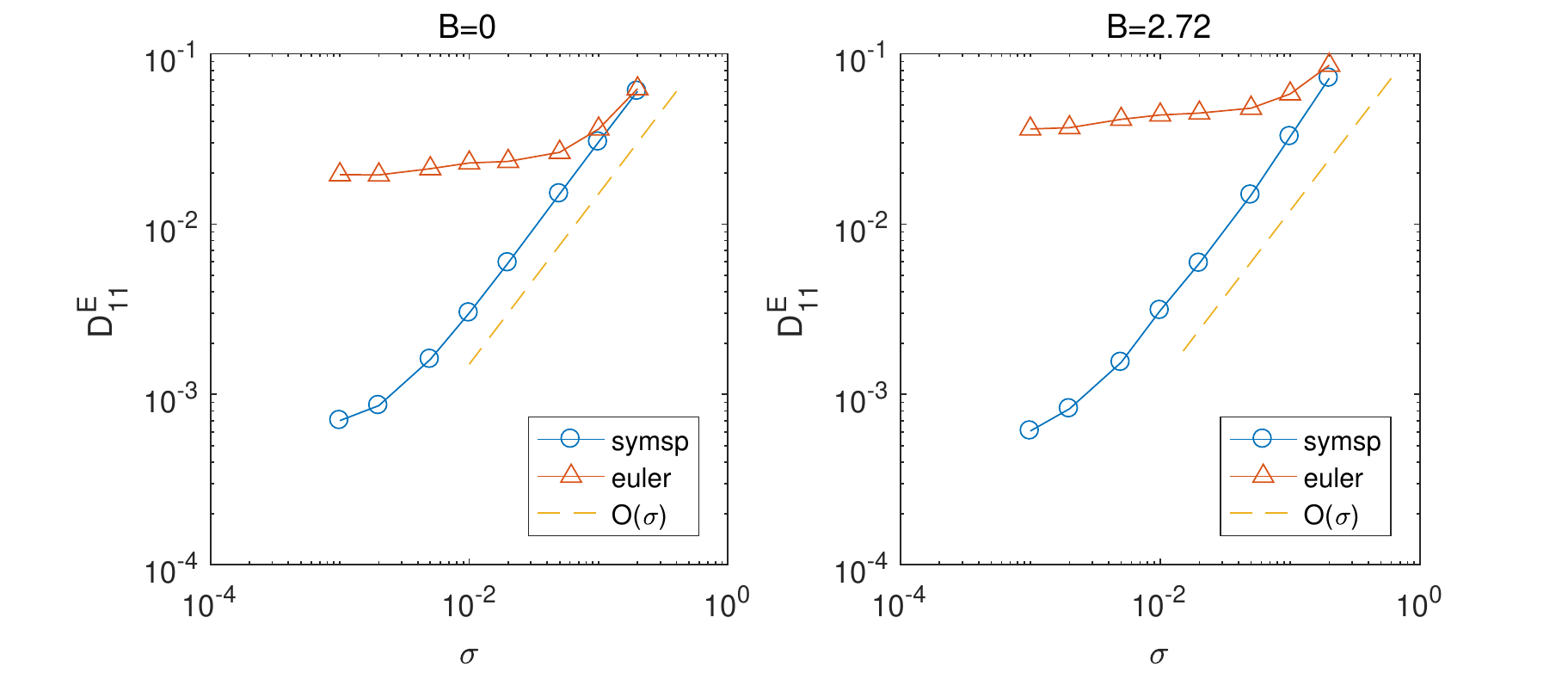}
	\caption{Numerical result for $D^E_{11}$, along with different $\sigma$}
	\label{fig:eg1f1}
\end{figure}
\begin{figure}[htbp]
	\centering
	\includegraphics[width=0.9\linewidth]{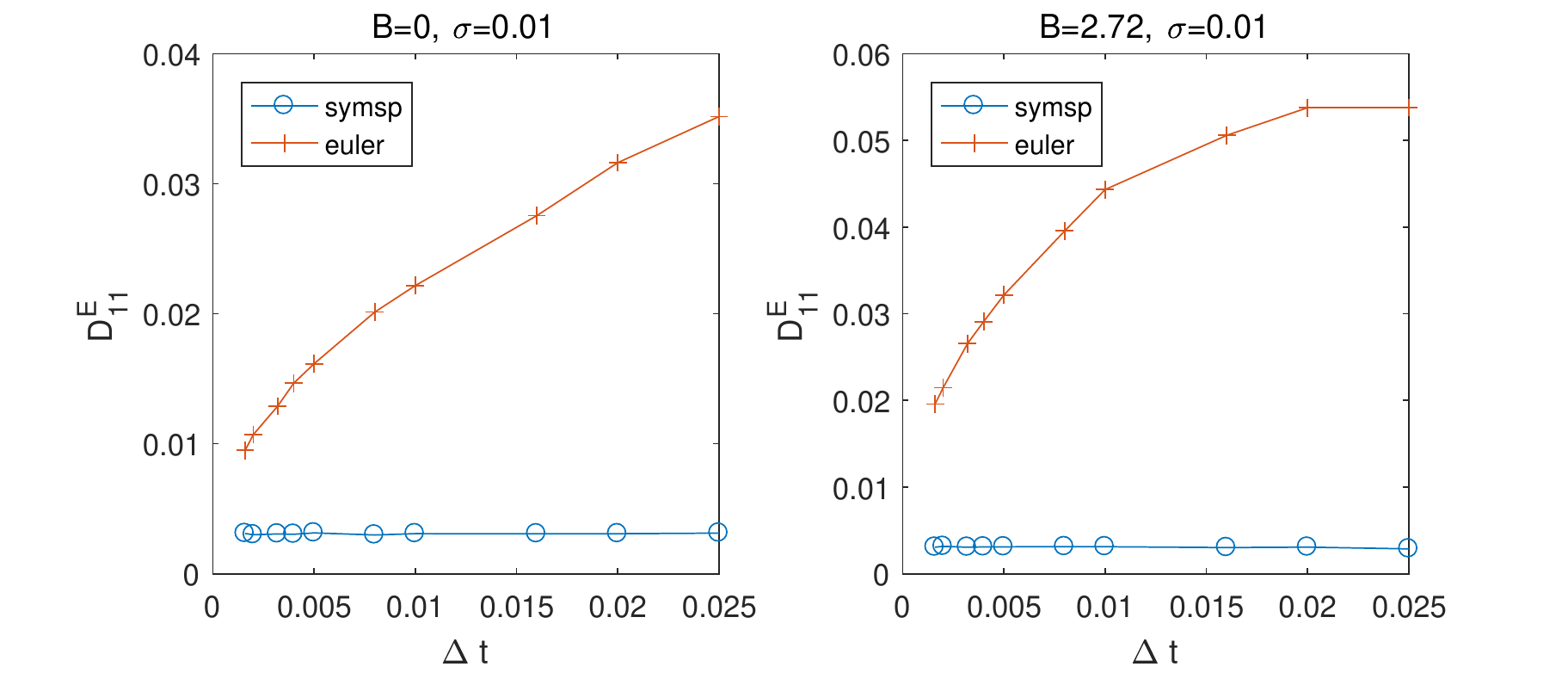}
	\caption{Numerical result for $D^E_{11}$, along with different $\Delta t$}
	\label{fig:eg1f2}
\end{figure}

\begin{remark}
We also tested a time-dependent Taylor-Green velocity field, which is generated by the Hamiltonian defined as
$H(t,p,q)= \frac{1}{k}\big(1+B\sin(\omega t)\big)\cos(kp)\sin(kq)$. This field can be used to model particle motion in
the ocean and in the atmosphere since it contains both vortices (convection cells) and linear uprising/sinking regions.
Our numerical results indicate that the asymptotic behaviours of effective diffusivity $D^E \sim \sigma I_2$. Namely,
there does not exist residual diffusivity for this time-dependent Taylor-Green velocity field.
\end{remark}

\subsection{Investigating residual diffusivity}
\noindent
We now turn to another chaotic cellular flow which is generated from a Hamiltonian defined as $H(t,p,q)=\big(\sin(p)-\sin(q)\big)+\theta \cos(t)\big(\cos(q)-\cos(p)\big)$. Then the particle path
satisfies the following SDE,
\begin{equation}\label{eqn:TimePeriodicCellularFlow}
\begin{cases}
dp=\big(\cos(q)+\theta \cos(t)\sin(q)\big)dt+\sigma dW_1,\\
dq=\big(\cos(p)+\theta \cos(t)\sin(p)\big)dt+\sigma dW_2.
\end{cases}
\end{equation}
The flow in Eq.\eqref{eqn:TimePeriodicCellularFlow} is fully chaotic (well-mixed at $\theta=1$). The first term of the velocity field
$\big(\cos(q),\cos(p)\big)$ is a steady cellular flow, but the second term of the velocity field $\theta \cos(t)\big(\sin(q),\sin(p)\big)$ is a time periodic perturbation that introduces an increasing amount of disorder in the flow trajectories as $\theta$ increases.

The flow in Eq.\eqref{eqn:TimePeriodicCellularFlow} has served as a model of chaotic advection for
Rayleigh-B\'enard experiment \cite{GetlingRayleigh:1998}. This type of flow has been investigated numerically in \cite{JackXinLyu:2017} by solving the cell problem Eq.\eqref{FokkerPlanck-cellproblem}. It was found that $D^E_{11} = O(1)$ as
$D_0\downarrow 0$, which implies the existence of the residual diffusivity. However, the solutions of the advection-diffusion equation Eq.\eqref{FokkerPlanck-cellproblem} develop sharp gradients as $D_0\downarrow 0$ and demand a large amount of computational costs. We shall show that our numerical method gives comparable results with far less computational costs.

In our numerical experiments, we choose time step $\Delta t=5\times10^{-2}$ and final time $T=5\times10^3$ in our symplectic scheme as smaller values of $\Delta t$ and larger values of $T$ do not alter the results significantly. We use $N_{mc}=5000$ independent
Monte Carlo sample paths to discretize the Brownian motions $dW1$ and $dW2$.

In Tab.\ref{tab:Jack_symsp}, we show the numerical results of $D^{E}_{11}$ for different $D_0$ and $\theta$. We also show the results in Fig.\ref{fig:eg2f2}. We observed a \emph{nonmonotone} dependence of $D^E_{11}$ vs. $\theta$ in the small $D_0$ regime, though the overall trend is that $D^E_{11}$ increases with the amount of chaos in the flows. Our numerical results again imply the existence of residual diffusivity for this type of chaotic flow.  As suggested in our previous numerical investigation, the Euler-Maruyama scheme needs a much finer time step to compute the residual diffusivity and the numerical results can be polluted by the diffusion of the scheme. Therefore, we do not test the Euler-Maruyama scheme in this experiment.

\begin{table}[htbp]
	\centering
\begin{tabular}{|c|c|c|c|c|c|c|}
 \hline
	 $\theta$ & $D_0 = 10^{-6} $ & $D_0 = 10^{-5} $ &$ D_0 = 10^{-4} $ & $D_0 = 10^{-3} $ & $ D_0 =10^{-2} $ & $D_0 = 10^{-1} $ \\
 \hline
	0.1   & 0.111547 & 0.084047 & 0.068833 & 0.072755 & 0.157947 & 0.504085 \\
	0.2   & 0.176780 & 0.161091 & 0.159181 & 0.169005 & 0.213418 & 0.547745 \\
	0.3   & 1.187858 & 0.901204 & 0.521761 & 0.356920 & 0.314840 & 0.550539 \\
	0.4   & 0.457187 & 0.453117 & 0.368187 & 0.385328 & 0.422116 & 0.538405 \\
	0.5   & 0.339372 & 0.352455 & 0.326034 & 0.361473 & 0.424855 & 0.645214 \\
	0.6   & 0.268441 & 0.246738 & 0.236696 & 0.256992 & 0.394480 & 0.704883 \\
	0.7   & 0.174016 & 0.169134 & 0.176643 & 0.215472 & 0.413941 & 0.754199 \\
	0.8   & 0.677995 & 0.605287 & 0.606582 & 0.516210 & 0.533211 & 0.796788 \\
	0.9   & 1.357033 & 1.363832 & 1.373394 & 1.084116 & 0.913423 & 0.908773 \\
 \hline
\end{tabular}
\caption{Numerical results of $D^E_{11}$ by the symplectic splitting scheme.}
	\label{tab:Jack_symsp}
\end{table}%

\begin{figure}[htbp]
\centering
\includegraphics[width=0.75\linewidth]{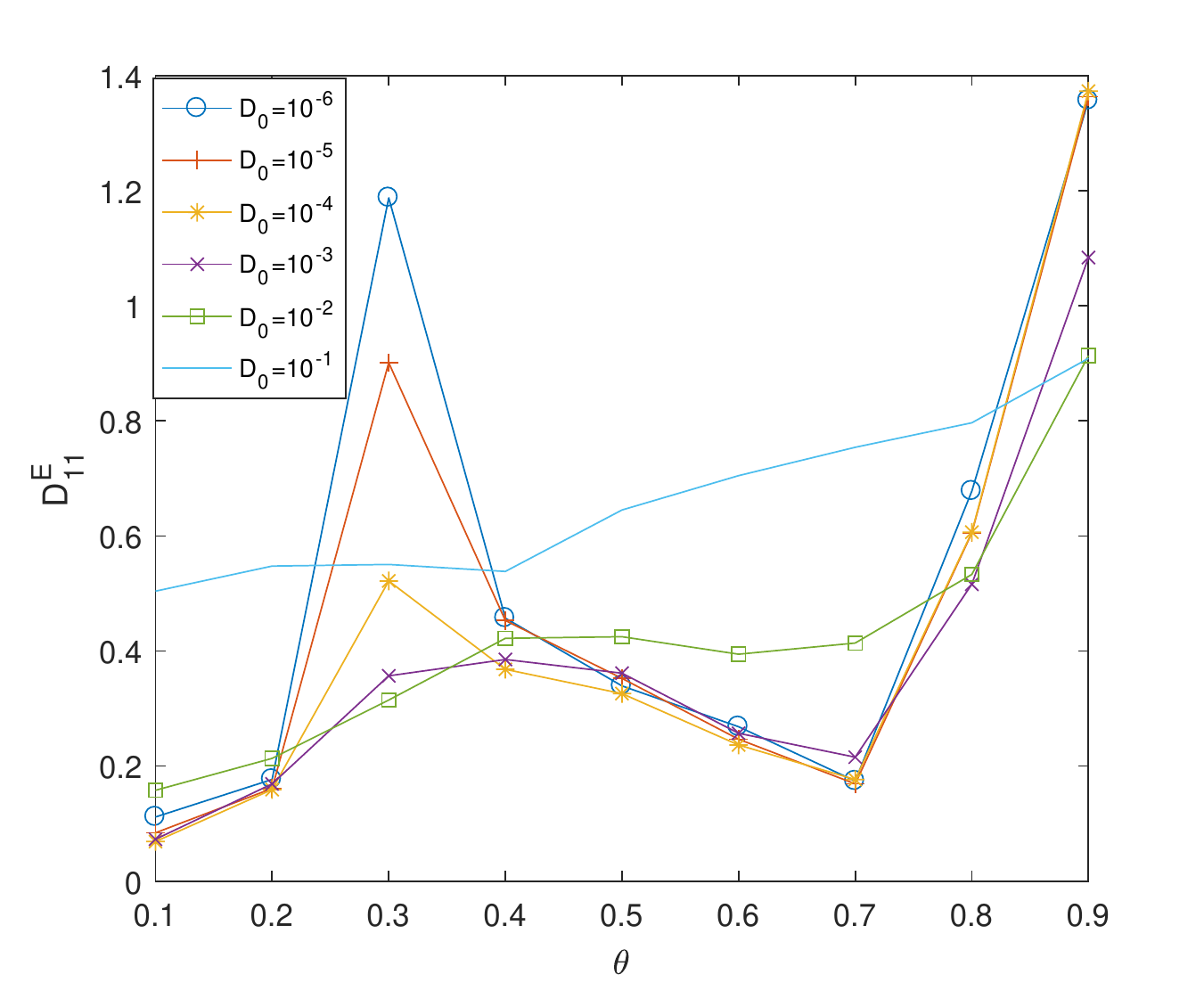}
\caption{The residual diffusivity results. $D^{E}_{11}$ vs. $\theta$ for the fully chaotic flow defined in Eq.\eqref{eqn:TimePeriodicCellularFlow}.}
\label{fig:eg2f2}
\end{figure}
\subsection{Investigating Stochastic flows}
\noindent
We are also interested in investigating the existence of the residual diffusivity for stochastic flows.
The homogenization of time-dependent random flows had been studied in literatures. Under certain integrability condition, it is proved that the effective diffusivity exists for the long-time large scale behavior of the solutions \cite{Fannjiang:97,Yaulandim:1998}. However, there are few numerical experiments to investigate effective diffusivity quantitatively
We shall use our symplectic splitting scheme to compute the effective diffusivity for stochastic flows.
More theoretical study will be reported in our subsequent paper.

The stochastic flow is constructed from the fully chaotic flow in Eq.\eqref{eqn:TimePeriodicCellularFlow}, where the time periodic function $\cos(t)$ is replaced by  an Ornstein-Uhlenbeck (OU) process $\eta_t$ \cite{Uhlenbeck:1930}. The OU process satisfies,
\begin{equation}\label{OUprocess}
d\eta_t=\theta_{ou}(\mu_{ou}-\eta_t)dt+\sigma_{ou}dW_t.
\end{equation}
where $\theta_{ou}>0$, $\mu_{ou}$, and $\sigma_{ou}>0$ are parameters and $dW_t$ denotes a Wiener process.
Specifically, $\theta_{ou}$ controls the speed of reversion, $\mu_{ou}$ is the long term mean level, and $\sigma_{ou}$ is the volatility or diffusion strength. In our numerical experiments, we choose $\mu_{ou}=0,\ \theta_{ou}=1$, and $\sigma_{ou}=1$, so
that the OU process has zero mean and the stationary variance is $\frac{\sigma_{ou}^2}{2\theta_{ou}}=\frac{1}{2}$. We choose the parameters in the OU process in such as way that its qualitative behavior is the same as $\cos(t)$. The particle path satisfies the following SDE,
\begin{equation}\label{StochasticCellularFlow}
\begin{cases}
dp=(\cos(q)+\theta\ \eta_t\sin(q))dt+\sigma dW^1_t\\
dq=(\cos(p)+\theta\ \eta_t\sin(p))dt+\sigma dW^2_t.
\end{cases}
\end{equation}
where the Brownian motions $dW^1_t$ and $dW^2_t$ are independent from the one used in the definition of the OU process Eq.\eqref{OUprocess}.

Since OU process has ergodic property, we choose a small amount of sample paths, say $n_{ou}=40$, and final computational time $T=5\times10^3$ to compute the effective diffusivity. In Tab.\ref{tab:OU}, we show the numerical results of $D^{E}_{11}$ for different $D_0$ and $\theta$, where each $D^E_{11}$ is the average values obtained from the $n_{ou}$ paths. In Fig.\ref{fig:eg2f4}, we show the results corresponding to Tab.\ref{tab:OU}. We observed a nonmonotone dependence of $D^E_{11}$ vs $\theta$ in time periodic cellular. Our numerical results again imply the existence of residual diffusivity for this type of stochastic flow. We observe however that the non-monotonic dependence in $\theta$ disappears. Namely, the residual diffusivity is an increasing function of $\theta$. Such phenomenon is due to the absence of resonance in stochastic flows.
Furthermore, we show the ergodicity results of the effective diffusivity in Fig.\ref{fig:eg2f5}. In this test, we
choose the parameters $\theta=0.1$, $D_0=10^{-2}$ and compute the effective diffusivity along $2000$ OU path. We show the histogram
of $D^E(\omega_{OU})$ at $T=100$, $T=200$, $T=500$, $T=5000$, and $T=20000$ respectively. This figure illustrates two facts:
firstly, as the computational time become long enough the histogram appears to converge to a limiting distribution. The limiting
distribution has much smaller variance and is centered closer to $0.156084$. Secondly, in the Tab.\ref{tab:Jack_symsp} we show
the residual diffusivity obtained from the fully chaotic (well-mixed) flow. When the parameters $\theta=0.1$, $D_0=10^{-2}$,
the corresponding  residual diffusivity is $D^{E}_{11}=0.157947$. Thus, the chaotic and stochastic flows may share some similar mechanism in long time behavior. More theoretic and numerical investigations will be studied in our future work.

\begin{table}[htbp]
	\centering
	\begin{tabular}{|c|c|c|c|c|c|c|}
\hline
		$\theta$ & $ D_0= 10^{-6} $ & $D_0= 10^{-5} $ &$ D_0= 10^{-4} $ & $ D_0=10^{-3} $ & $D_0= 10^{-2} $ & $ D_0= 10^{-1} $ \\
\hline
0.1   & 0.036442 & 0.037821 & 0.042649 & 0.064412 & 0.156084 & 0.485647 \\
0.2   & 0.070701 & 0.074095 & 0.075525 & 0.094416 & 0.172281 & 0.491868 \\
0.3   & 0.106238 & 0.104986 & 0.112149 & 0.123868 & 0.195421 & 0.496326 \\
0.4   & 0.137335 & 0.141704 & 0.145786 & 0.154876 & 0.221186 & 0.513384 \\
0.5   & 0.171326 & 0.173708 & 0.176357 & 0.187868 & 0.252861 & 0.522133 \\
0.6   & 0.197188 & 0.200511 & 0.205098 & 0.220810 & 0.272689 & 0.539465 \\
0.7   & 0.232775 & 0.231468 & 0.240672 & 0.248353 & 0.314599 & 0.563992 \\
0.8   & 0.259921 & 0.255478 & 0.268048 & 0.280238 & 0.332105 & 0.589805 \\
0.9   & 0.286707 & 0.291560 & 0.290207 & 0.294778 & 0.365502 & 0.605338 \\
\hline
	\end{tabular}%
\caption{Numerical results of $D^E_{11}$ by the symplectic splitting scheme. The flow is defined by OU process.}
\label{tab:OU}
\end{table}
\begin{figure}[htbp]
\centering
\includegraphics[width=0.75\linewidth]{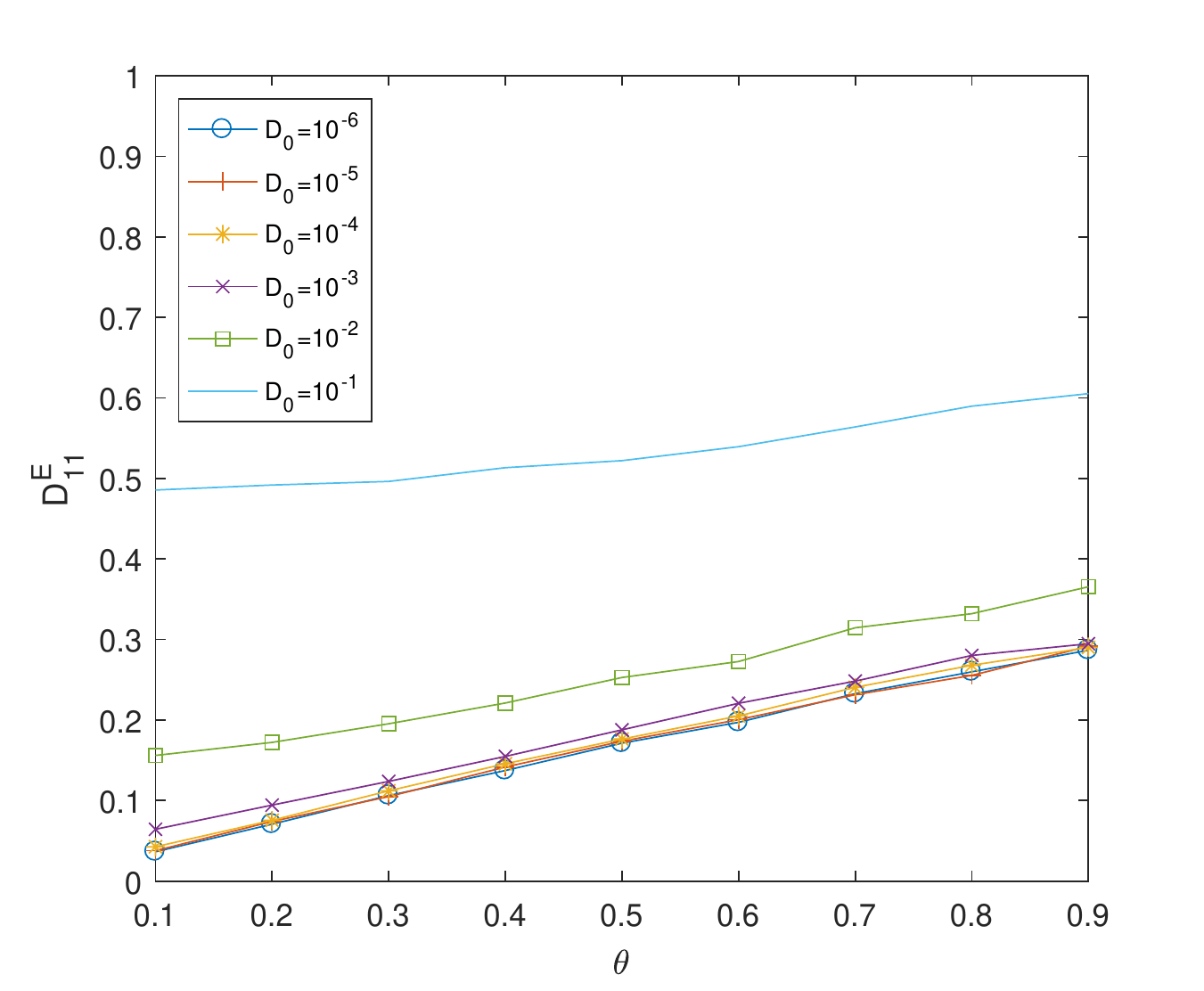}
\caption{The residual diffusivity results. $D^{E}_{11}$ vs. $\theta$ for the Stochastic flow driven by an OU process defined in Eq.\eqref{StochasticCellularFlow}.}
\label{fig:eg2f4}
\end{figure}

\begin{figure}[htbp]
\centering
\includegraphics[width=0.75\linewidth]{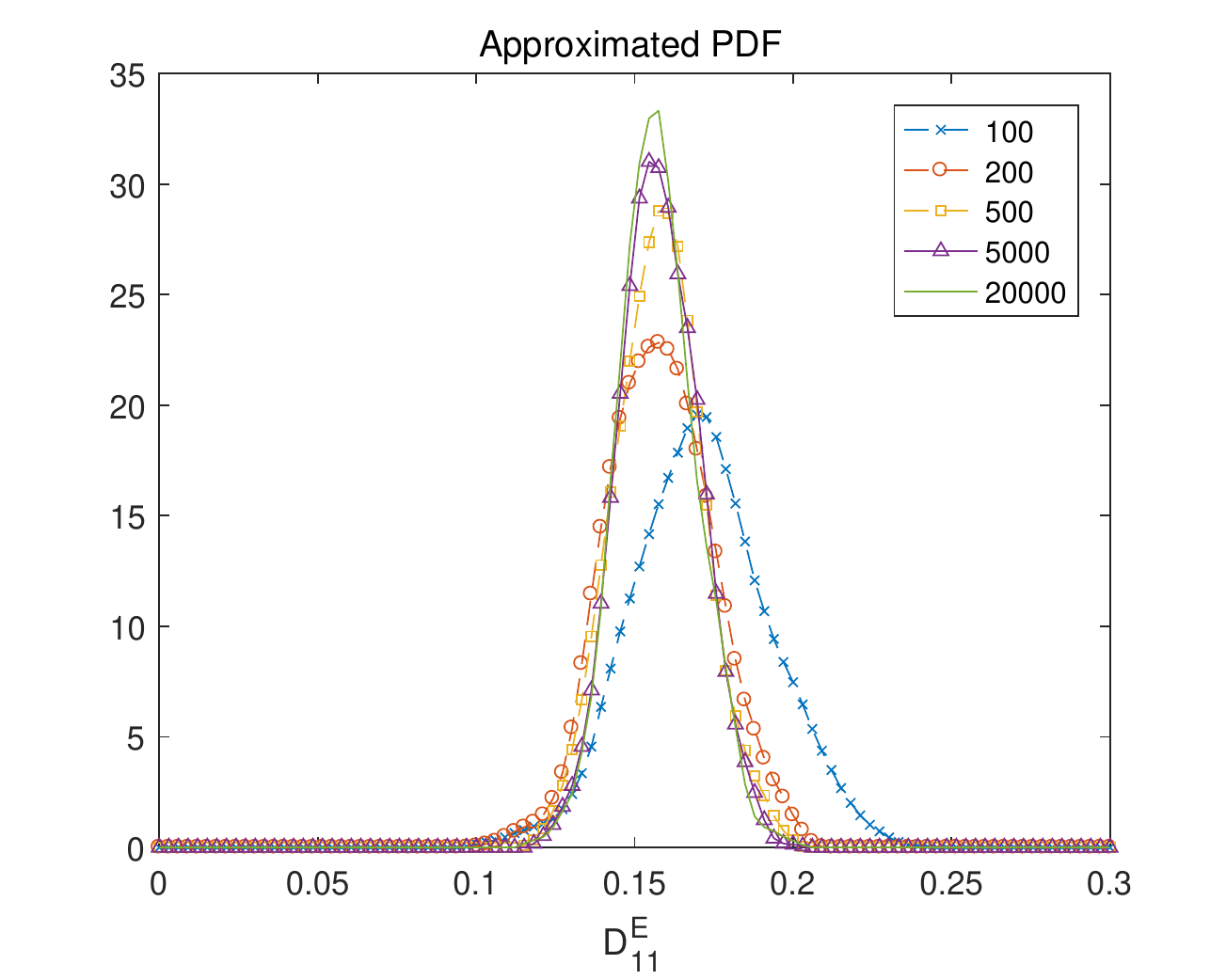}
\caption{Histogram of the residual diffusivity results. $D^{E}_{11}$ for the Stochastic flow driven by an OU process defined in Eq.\eqref{StochasticCellularFlow} that are computed at different final times.}
\label{fig:eg2f5}
\end{figure}

\subsection{Behavior of the long-time integration}\label{sec:LongTimeIntegration}
\noindent
Theorem \ref{thm:FollowAsymptoticHamiltonian} proves that the symplectic splitting scheme preserves the asymptotic Hamiltonian structure that enables us to compute the stable long-time behaviour of the effective diffusivity of chaotic and stochastic flows.  We now keep using the flow Eq.\eqref{eqn:TimePeriodicCellularFlow} and compute a much longer time solution with final time $T=5\times10^5$.

\begin{figure}[htbp]
	\centering
	\includegraphics[width=0.75\linewidth]{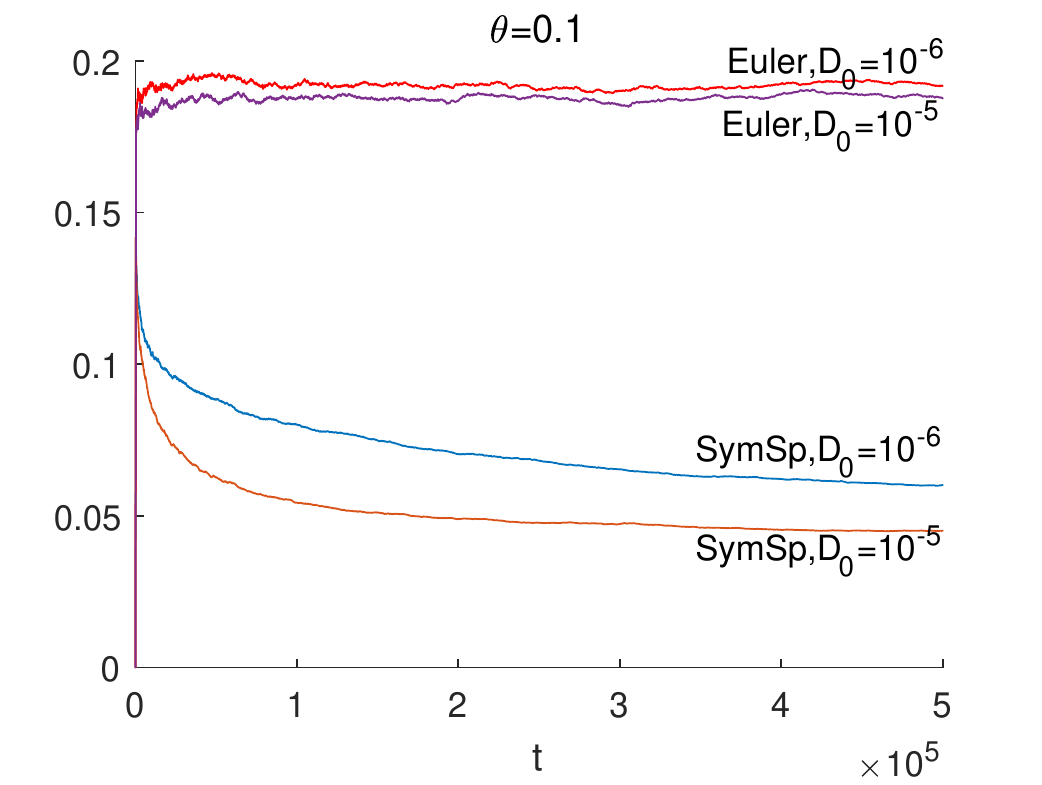}
	\caption{Behavior of $\frac{<(x_1(t)-x_1(0))^2>}{2t}$ as a function of time for two different methods. }
	\label{fig:eg2f1}
\end{figure}
In Figure \ref{fig:eg2f1}, we plot the effective diffusivity $D^E_{11}$ as a function of time obtained using different methods and parameters. The top two lines correspond to the Euler-Maruyama method for $\sigma=10^{-5}$ and $\sigma=10^{-6}$, while the bottom
two lines correspond to the symplectic splitting method. It is clear that results obtained from symplectic splitting method converge to a more stable value. A probable explanation is that modified flow of Euler method is not divergence-free while
the solution obtained using the symplectic splitting scheme follows an asymptotic Hamiltonian. This is proved in our Theorem \ref{thm:FollowAsymptoticHamiltonian}.

Another evidence comes from Figure \ref{fig:eg2f3}, where we plot the phase plane for two different numerical methods. The realization of the noise is the same and we integrated up to time $T=10^3$ with time step $\Delta t=10^{-2}$. We choose the parameters $\theta=0.1$ and $D_0=10^{-5}$. From these results, we find that the paths oscillate near a line with slope $1$. It is clear that the behavior of the particle is drastically different. In the case of Euler-Maruyama method the particle appears to be much more diffusive than in the case of the symplectic splitting scheme.

\begin{figure}[htbp]
	\centering
	\includegraphics[width=0.90\linewidth]{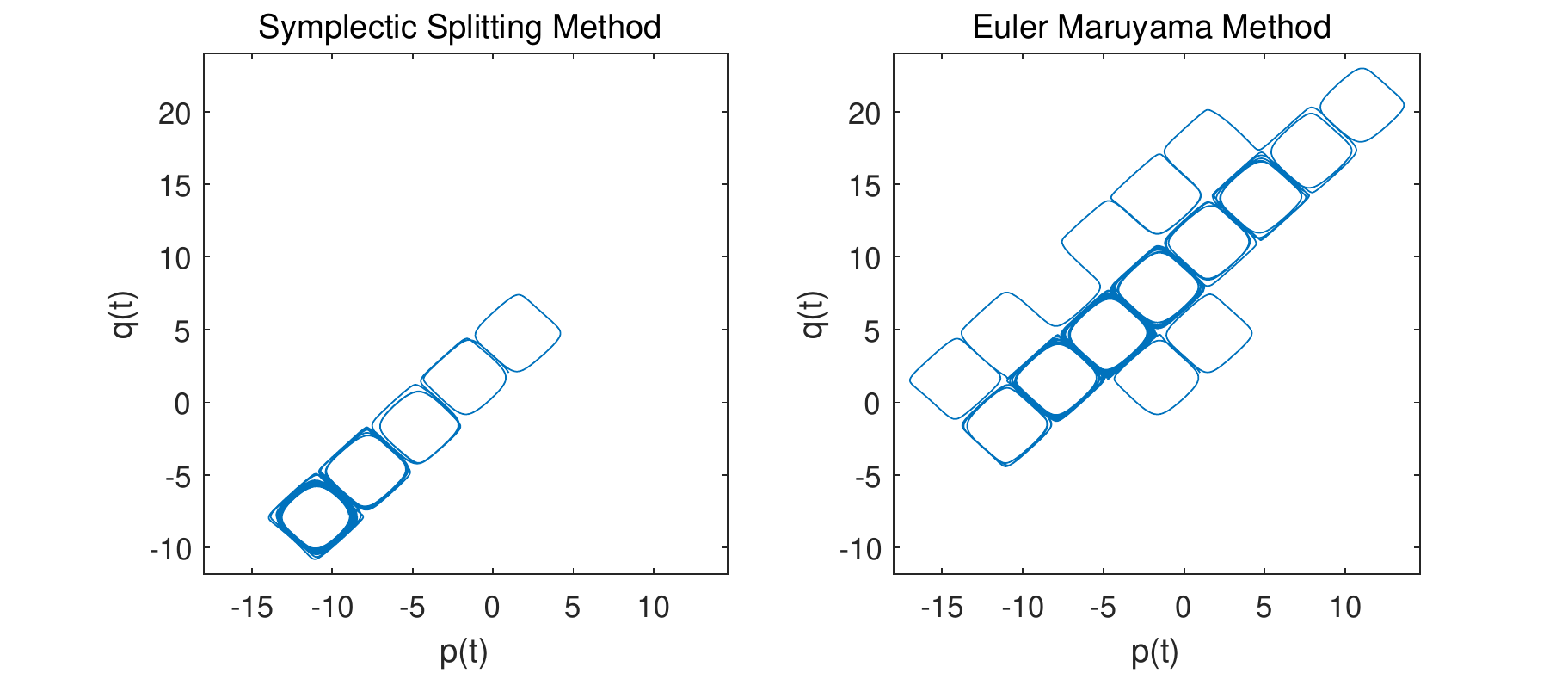}
	\caption{Phase plane for the two different methods.}
	\label{fig:eg2f3}
\end{figure}

In Figure \ref{fig:eg5f1}, we show how the modified equation approximate the original problem, where we consider the chaotic cellular flow \eqref{eqn:TimePeriodicCellularFlow}. More specifically, we plot the effective diffusivity $2(D^{E}_{11}+D^{E}_{22})$ as of function of time obtained using different methods and we choose the parameter $\theta=0.1$ and $D_0=10^{-5}$. From our numerical results, we find that the effective diffusivity obtained using our method with time step $dt=0.05$ agrees very well that one obtained from solving the modified equation using the Euler-Maruyama method with time step $dt=0.002$. Namely, we approximately achieve a $25X$ speedup over the  Euler-Maruyama method. The Euler method with $dt=0.05$ also generates results that agrees with its corresponding modified equation with finer time step. But the effective diffusivity converges to the wrong result.

\begin{figure}[htbp]
	\centering
	\includegraphics[width=0.75\linewidth]{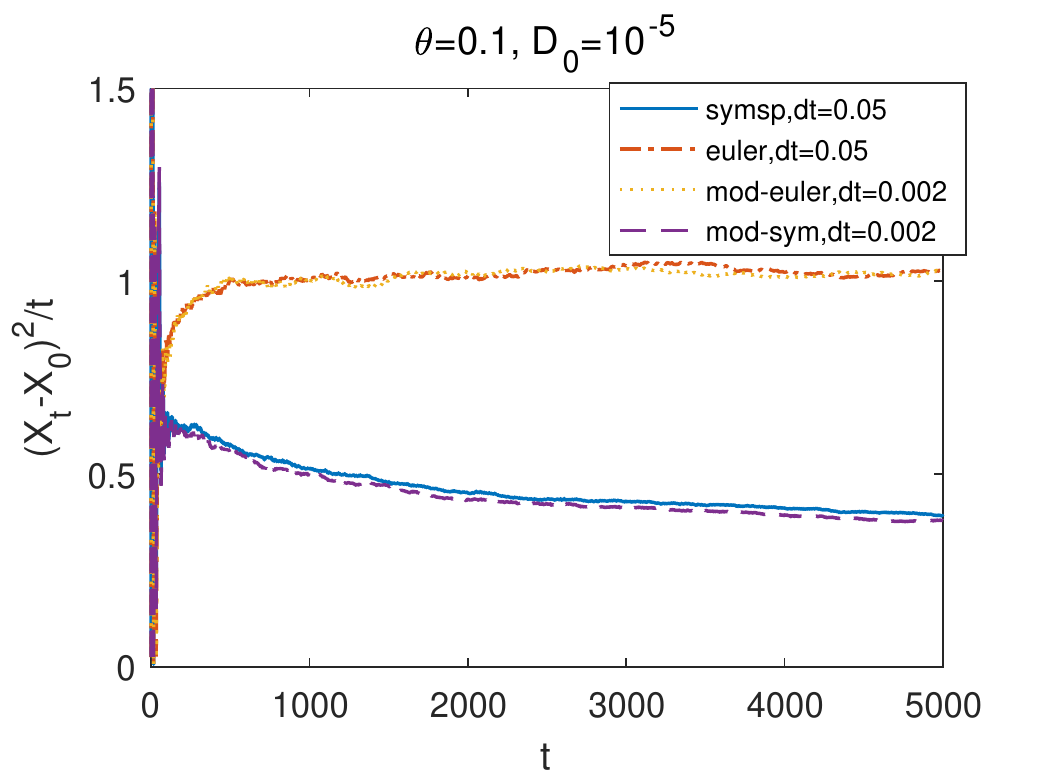}
	\caption{Behavior of $\frac{<(x_1(t)-x_1(0))^2>+<(x_2(t)-x_2(0))^2>}{t}$ for two different methods.}
	\label{fig:eg5f1}
\end{figure}

\section{Conclusions}
\noindent
Quantifying diffusion enhancement in fluid advection is a fundamental problem that has many applications in physical and engineering sciences. We proposed a class of structure preserving schemes that can efficiently compute the effective diffusivity of chaotic and stochastic flows containing complex streamlines.  In addition, we investigate the existence of the  residual diffusion phenomenon in chaotic and stochastic advection, which is an interesting problem by itself.  The effective diffusivity as well as the residual diffusivity can be computed by solving the Fokker-Planck equation in the Eulerian formulation. However, when the diffusion coefficient becomes small, the solutions of the advection-diffusion equation develop sharp gradients and thus demand a large amount of computational costs.

We compute the effective diffusivity in the Lagrangian formulation, i.e., solving SDEs.  We split the original problem into a deterministic sub-problem and a random perturbation, where the former is discretized using a symplectic preserving scheme while
the later is solved using the Euler-Maruyama scheme. We provide rigorous error analysis for our new numerical integrator using the backward error analysis technique and show that our method outperform standard Euler-based integrators. Numerical results are presented to demonstrate the accuracy and efficiency of the proposed method for several typical chaotic and stochastic flows problems of physical interests. We find that the residual diffusivity exists in some time periodic and stochastic cellular flows.

\section{Acknowledgements}
\noindent
The research of Z. Wang is partial supported by the Hong Kong PhD Fellowship Scheme.
The research of J. Xin is supported by NSF grants DMS-1211179, DMS-1522383. The research of Z. Zhang is supported by Hong Kong RGC
grants (Project 27300616 and 17300817), National Natural Science Foundation of China (Project 11601457), Seed Funding Programme for Basic Research (HKU), and the Hung Hing Ying Physical Sciences Research Fund (HKU).

\section*{References}
\bibliographystyle{plain}

\end{document}